\documentclass[a4paper,11pt]{article}
\usepackage{geometry}
\usepackage[ansinew]{inputenc}
\usepackage{amssymb}
\usepackage{amsmath}
\usepackage{amsthm}
\usepackage{bbm}
\usepackage{mathrsfs}
\usepackage{enumerate}
\usepackage{graphicx}
\usepackage{comment}

\usepackage[authoryear, sort&compress]{natbib}
\input{math.sty}

\begin{document}

\title{High-frequency Donsker theorems for L\'evy measures }

\author{Richard Nickl~*, ~~~ Markus Rei\ss ~$\dagger$, ~~~Jakob S\"ohl~*,~~~Mathias Trabs $\dagger$  \\
\\
\textit{University of Cambridge \footnote{Statistical Laboratory, Department of Pure Mathematics and Mathematical Statistics, University of Cambridge, CB30WB, Cambridge, UK. Email: r.nickl@statslab.cam.ac.uk, j.soehl@statslab.cam.ac.uk}} ~ and  \textit{ Humboldt-Universit\"at zu Berlin \footnote{Institut f\"ur Mathematik, Humboldt-Universit\"at zu Berlin, Unter den Linden 6, 10099 Berlin, Germany. Email: mreiss@math.hu-berlin.de, trabs@math.hu-berlin.de }}}

\maketitle

\begin{abstract}
Donsker-type functional limit theorems are proved for empirical processes arising from discretely sampled increments of a univariate L\'evy process. In the asymptotic regime the sampling frequencies increase to infinity and the limiting object is a Gaussian process that can be obtained
from the composition of a Brownian motion with a covariance operator
determined by the L\'evy measure. The results are applied
to derive the asymptotic distribution of natural estimators for the
distribution function of the L\'evy jump measure. As an application we deduce
Kolmogorov-Smirnov type tests and confidence bands.

\medskip


\end{abstract}

\section{Introduction}

Suppose that $(L_t: t \ge 0)$ is a real-valued L\'evy process defined on some
probability space $(\Omega, \mathcal A, \Pr)$ and we observe $n$ of its
increments
\begin{equation} \label{data}
X_k = L_{k \Delta} - L_{(k-1)\Delta},\quad k =1, \dots, n,
\end{equation}
sampled at frequency $1/\Delta>0$. Equivalently the $X_k$'s are drawn i.i.d.~from some
infinitely divisible distribution $\PP_\Delta$, with corresponding empirical measures $\PP_{\Delta, n} = \frac{1}{n} \sum_{k=1}^n \delta_{X_k}$.

\smallskip

L\'evy processes are increasingly popular in stochastic modelling. A question of key importance is how the structure of the L\'evy process, particularly its jump behaviour, can be
recovered from these observed increments.
From a statistical point of view it is natural to consider a growing observation
horizon $n \Delta \to \infty$. If simultaneously $\Delta = \Delta_n$ approaches
zero one speaks of a `high-frequency' sampling regime, as opposed to
`low-frequency' sampling where $\Delta$ remains fixed.
Inference problems of this kind have recently gained increased attention.
\cite{JMV05} studied nonparametric inference for L\'evy-driven
Ornstein--Uhlenbeck processes. \cite{BR06} treat nonparametric estimation of
L\'evy processes in a financial model.
Low-frequency observations were considered, e.g., by
\cite{neumannReiss2009}, \cite{B10}, \cite{G12} as well as
\cite{NicklReiss2012}, whereas
\cite{figueroa-lopez2009,Figueroa-Lopez2011} treats high-frequency
observations.
Nonparametric estimation of L\'evy processes in a model selection context was
studied by
\cite{CGC11} and
\cite{kappus2014}. A general discussion of the literature
and further references can be found in the recent survey paper \cite{R13}.

\smallskip

By the L\'evy--Khintchine representation (\cite{sato1999}) the L\'evy process $(L_t: t \ge 0)$ is entirely
characterised by three parameters: the \textit{diffusion coefficient} $\sigma^2$
describing the Brownian motion component, the \textit{centring} or
\textit{drift} parameter $\gamma$, and the \textit{L\'evy measure} $\nu$.
Recovering the L\'evy process can thus be reduced to recovering the L\'evy
triplet $(\sigma^2, \gamma, \nu)$. Statistical inference for the one-dimensional
parameters $\sigma^2, \gamma$ can be based on standard statistics such as
the quadratic variation and the sample average of the increments, or on spectral
estimators, see Section~\ref{discussion} for discussion and references.

\smallskip

An intrinsically more complex problem than inference on $\sigma^2$ and $\gamma$
is the recovery of the L\'evy measure $\nu$, which describes the jump behaviour
of the L\'evy process. We recall that there is a bijection between the set of
L\'evy measures $\nu$ and all positive Borel measures $\nu$ on $\R$ s.t.
\[\int_{\R} (1 \wedge x^2) \nu(\d x) <\infty, \quad \nu(\{0\})=0.\]
Thus a natural target is to recover the cumulative distribution function
\begin{equation}\label{eq:cdf}
N(t) 
= \int_{-\infty}^t (1 \wedge x^2) \nu(\d x), ~~t \in \R,
\end{equation}
from the observed increments; it encodes both
local and global information about $\nu$.
The presence of $(1\wedge x^2)$
smooths the singularity
that $\nu$ may possess at the origin.
Other possibilities to smooth the singularity exist and our results will cover
functions from a general
class (see Section~\ref{secProc}).
In particular this will include recovery of the
distribution function
\begin{equation}
\mathcal N(t) =\int_{-\infty}^t \nu(\d x), ~t < 0,\quad \text{and} \quad
\mathcal N(t) =\int_{t}^{\infty} \nu(\d x), ~t > 0,
\end{equation}
of the L\'evy measure at any point $t\neq0$.

For statistical applications, inference on the functions $N, \mathcal N$ in the
uniform norm $\|\cdot\|_\infty$ on the real line is of particular interest,
paralleling the classical Donsker-Kolmogorov-Smirnov central limit theorems
$$\sqrt n (F_n -F) \to^{\mathcal L} \mathbb G_F $$ in the space of bounded functions on $\R$, where $F_n$ is the empirical
distribution function of a random sample from distribution $F$, and where
$\mathbb G_F$ is the $F$-Brownian bridge (\cite{dudley1999,
vanderVaartWellner1996}).  In the L\'evy setting, \cite{NicklReiss2012} considered an estimator for the distribution function
$\mathcal N(t)$, $|t|\ge\zeta,$  based on low-frequency observations ($\Delta$ fixed) and proved such a Donsker-Kolmogorov-Smirnov theorem.
The purpose of the present article is to derive such results when also $\Delta
\to 0$. The main message is that high-frequency observations reveal much finer
statistical properties of the L\'evy measure, and inference is possible for a
much larger class of L\'evy processes than considered in \cite{NicklReiss2012},
including processes with a nonzero Gaussian component. Moreover, the theory does
not only cover nonlinear `inversion' estimators based on the L\'evy-Khintchine
formula, but also `linear' estimators  based on elementary counting statistics.
At the heart of these results is a general purpose uniform central limit theorem
for a basic `smoothed empirical process' arising from the $X_k$'s in
(\ref{data}), see  Theorem \ref{main1} below.

\smallskip

In the next section we introduce the estimators and give the main results as
well as some statistical applications. In
Section~\ref{secProc} we show how to reduce the proofs to the study of a unified smoothed empirical process, and in Section~\ref{discussion} we discuss
our conditions and their interpretation in a variety of concrete examples of
L\'evy processes. The remainder of the article is then devoted to the proofs of
our results.

\section{Main results: Asymptotic inference on the L\'evy measure $\nu$}\label{secEstimators}

In this section we study two approaches to estimate the distribution
functions $N, \mathcal N$ of a L\'evy measure, based on discrete observations (\ref{data}). The first estimator is constructed by a direct
approach and counts the number of increments below a certain threshold, where
increments are weighted by $1\wedge X_k^2$. The second approach
relies on the L\'evy--Khintchine representation and a
spectral regularisation step.

\subsection{Basic notation and assumptions}

The symbol $\ell^\infty(T)$ denotes the space of bounded functions on a set $T$ normed by the usual supremum norm $\|\cdot\|_\infty$. We will measure the smoothness of functions in a local H\"older
norm: Denoting by $C(U)=C^0(U)$ the set of all functions on an open
set $U\subset\R$ which are  bounded, continuous and real-valued, we
define for $s>0$ the H\"older spaces
\[
  C^s(U):=\Big\{f\in C(U): \|f\|_{C^s(U)}:=\sum_{k=0}^{\lfloor
s\rfloor}\sup_{x\in U}|f^{(k)}(x)|+\sup_{x,y\in U:x\neq y} \frac{|f^{(\lfloor
s\rfloor)}(x)-f^{(\lfloor s\rfloor)}(y)|} {|x-y|^{ s-\lfloor
s\rfloor}}<\infty\Big\}
\]
where $\lfloor  s \rfloor$ denotes the largest integer strictly smaller than
$s$.

\smallskip

We assume throughout this article that the L\'evy measure has finite second moments,
\begin{equation}\label{eqSecMom}
 \int_{\R} x^2 \nu(\d x) <\infty.
\end{equation}
This is equivalent to $\PP_\Delta$ having finite second moments $\forall \Delta>0$ (\cite{sato1999}).

For our main results we will rely on the following stronger assumption on $\nu$. Slightly abusing notation we shall use the same symbol for a measure and its Lebesgue density, if the latter exists. Also we use
$\lesssim, \gtrsim, (\sim)$ to denote (two-sided)
inequalities up to a multiplicative constant.

\begin{assumption}\label{assProc}
  \begin{enumerate}
  \item[(a)] For some $\eps>0$ we have $$\int_{\R}|x|^{4+\eps}\nu (d
x)<\infty.$$
  \item[(b)] The L\'evy measure $\nu$ has a Lebesgue density, also denoted by $\nu$, and $$(1\wedge x^4)\nu\in \ell^\infty(\R).$$
  \item[(c)] The measure $x^3\PP_{\Delta}$ admits a Lebesgue density, also denoted by $x^3\PP_\Delta$, satisfying, as $\Delta \to 0$,
$$\norm{x^3\PP_{\Delta}}_\infty\lesssim \Delta.$$
  \item[(d)] Let $U$ be a neighbourhood of the origin and $V \subset \R$. For some $s>0$ and some finite constants $c_t>0$, $t\in V$, we have
\begin{align*}
 \|g_t(-\cdot) \ast
(x^2\nu)\|_{C^s(U)} \le c_t \qquad \text{ with }
g_{t}(x):=
(1\wedge x^{-2})\1_{(-\infty,t]}(x).
\end{align*}
 \end{enumerate}
\end{assumption}

\medskip

Assumptions~(a) and (b) are a moment condition and a mild regularity condition
on the L\'evy measure, respectively. Assumption~(c) is
the key condition and will be discussed in detail in Section~\ref{sec:Boundx3P}. Here we just remark that for
instance under the assumption $x^3\nu\in \ell^\infty(\R)$, this condition will be shown
 to be satisfied whenever the diffusion coefficient is positive ($\sigma>0$).
Assumption~(d) is used to control approximation theoretic properties of the distribution function of $x^2\nu$. For global results ($V=\R$) we notice that it is easily seen that (d)
is satisfied with a uniform constant $c>0$ if $x^2\nu\in C^{s-1}(\R), s \ge 1$.

Recall that a function $l$ defined on $(0,\infty)$ is \emph{slowly varying} at
the origin if
\[\frac{l(tx)}{l(t)}\to1,\qquad \text{as } t\to0,~~\forall x>0.\]
A function $f$ is \emph{regularly varying} at the origin with exponent
$p\in\R$ if $f$ is of the form
\begin{align*}
 f(x)=x^p l(x)
\end{align*}
with $l$ slowly varying at the origin.
We denote the
symmetrised L\'evy density by $\tilde\nu(x):=\nu^+(x)+\nu^-(-x)$,
where $\nu^+=\nu\1_{\R^+}$ and
$\nu^-=\nu\1_{\R^-}$.

Throughout the paper we write $\to^\mathcal L$
to denote convergence in distribution of random elements in a
metric space as in Chapter 1 in \cite{vanderVaartWellner1996}.

\subsection{The direct estimation approach}

In the
high-frequency regime $\Delta \to 0$ inference on $\nu$ can be based on the
following simple observation.
\begin{lemma} \label{weakcon}
If the L\'evy measure $\nu$ satisfies \eqref{eqSecMom}, then we have weak
convergence
\begin{equation} \label{eqWeakConvNu}
x^2  \frac{\PP_\Delta}{\Delta}  \to \sigma^2 \delta_0 + x^2 \nu
\end{equation}
as $\Delta \to 0$ in the sense that
\begin{equation} \label{prettyweak}
\int_{\R} f(x) x^2 \frac{\PP_\Delta(\d x)}{\Delta} \to \sigma^2 f(0) + \int_{\R}
f(x) x^2 \nu(\d x)
\end{equation}
for every bounded continuous function $f: \R \to \R $.
\end{lemma}
Starting with L\'evy processes without
diffusion component, that is, with $\sigma=0$, the asymptotic identification
\eqref{eqWeakConvNu} motivates a linear estimator of
$N(t)$ given by
\begin{equation}\label{eq:naiveEst}
 \tilde N_n(t):=\int_{-\infty}^t(1\wedge x^2)\frac{\PP_{\Delta,n}(\d
x)}{\Delta}
=\frac1{n\Delta}\sum_{k=1}^{n}
(1\wedge X_k^2)\1_{(-\infty,t]}(X_k),\quad t \in \R,
\end{equation}
where $\PP_{\Delta,n}=\frac{1}{n}\sum_{k=1}^{n}\delta_{X_k}$ is the empirical
measure of the increments from (\ref{data}).

Similarly, and including the case $\sigma \neq 0$, one can estimate the function $\mathcal N$ by
\begin{align*}
    \tilde{\mathcal N}_n(t):=\int_ {\R} f_t(x)\frac{\PP_{\Delta,n}(\d
x)}{\Delta}\quad\text{with}\quad f_t(x):=\begin{cases}
                                           \1_{(-\infty,t]},\quad t<0\\
                                           \1_{[t,\infty)},\quad t>0.
                                         \end{cases}
\end{align*}

\smallskip

We start with a theorem for the basic estimator $\tilde N_n$.

\begin{theorem}
\label{thm:UCLTnaive}
Let $\sigma=0$ and grant Assumption~\ref{assProc} for $V=\R$ for some
$s\in (0,2]$ and with uniform constant $\sup_{t\in\R}c_t=c<\infty$.

Assume either that

a) the density of $x \nu$ exists and is of bounded variation, and
the drift $\gamma_0 := \gamma-\int x \nu (\d x)=0$;~ or that

b) $\tilde\nu(x)=\nu^+(x)+\nu^-(-x)$ is regularly varying at
zero with exponent $-(\beta+1)$, $\beta\in(0,2), s \in (0,2-\beta)$.

If $n\to\infty$ and
$\Delta_{n}\to0$ such that
\begin{align*}
 n\Delta_{n}\to\infty,\quad \Delta_{n}=o(n^{-1/(s+1)})\quad\text{ and }\quad
\log^4(1/\Delta_n)=o(n\Delta_n),
\end{align*}
then
\[
\sqrt{n\Delta_n}\big(\tilde{N}_{n}-N\big)\to^{\mathcal{L}}\mathbb{G}
\quad\text{in}\quad\ell^{\infty}(\R),
\]
where $\mathbb{G}$ is a tight Gaussian random variable arising from the
centred Gaussian process
$\{\mathbb{G}(t):t\in\R\}$
with covariance
\[
\E[\mathbb{G}(t)\mathbb{G}(t')]=\int_{-\infty}^{t \wedge t'}(1\wedge
x^4)\nu(\d x),\quad t,t'\in\R.
\]
\end{theorem}

Since estimation at the origin $t=0$ is included in the last theorem, the assumption $\sigma=0$ is natural  -- the simple linear estimator $\tilde N_n$ cannot distinguish between arbitrarily small jumps and a Brownian diffusion component. Moreover, setting the drift $\gamma_0 =0$ in a) rules out situations where the measure $\PP_\Delta$ has a discrete component $\delta_{\Delta \gamma_0}$, which causes complications in the analysis.  Simultaneous estimation of all parameters of the L\'evy triplet without restrictions on $\gamma$ and $\sigma$ will be considered by non-linear methods in the next subsection.

 The conditions a) and b) are required to show that the deterministic `bias' term arising from the basic
linear estimator is negligible in the limit distribution (Proposition
\ref{prop:bias}).  The case a) covers many examples of finite activity L\'evy processes as well as some limiting cases where the singularity of $\nu$ at the origin behaves like $|x|^{-1}$ (see Subsection \ref{secExamples} for examples). In contrast case b) covers infinite activity processes with a singularity of the form $|x|^{-1-\beta}, \beta \in (0,2)$. The assumption of regular variation of $\tilde \nu$ at zero is natural in all
key examples considered in Subsection \ref{secExamples} below -- typically the variation exponent will be closely related to the regularity $s$ of $N$, and we discuss in Section \ref{sec:discussion} how our
parameter constraints on $\beta$ and $s$ are compatible.

When the origin is \textit{excluded} from consideration, an argument of \cite{Figueroa-Lopez2011} can be used to obtain the following result for the linear estimator $\tilde {\mathcal N}$, allowing also for $\sigma \neq 0$:
\begin{theorem}\label{thm:UCLTnaive2}
Grant Assumptions~\ref{assProc}(a)-(c).
Let $\zeta>0$ and suppose that the L\'evy density $\nu$ is Lipschitz
continuous in an open set $V_0$ containing
$V=(-\infty,-\zeta]\cup[\zeta,\infty)$.
If $n\to\infty$ and $\Delta_{n}\to0$ such that
\begin{align*}
 n\Delta_{n}\to\infty,\quad \Delta_{n}=o(n^{-1/3})\quad\text{ and }\quad
\log^4(1/\Delta_n)=o(n\Delta_n).
\end{align*}
Then
\[
\sqrt{n\Delta_n}\big(\tilde{\mathcal
N}_{n}-\mathcal N\big)\to^{\mathcal{L}}\mathbb{W}
\quad\text{in}\quad\ell^{\infty}(V),
\]
where $\mathbb{W}$ is a tight Gaussian random variable arising from the
centred Gaussian process
$\{\mathbb{W}(t):t\in\R\}$
with covariance, for $f_t=\1_{(-\infty,t]}$ for $t<0$ and $f_t=\1_{[t,\infty)}$ for $t>0$,
\[
\E[\mathbb{W}(t)\mathbb{W}(t')]=\int_{\R}f_t(x)f_{t'}(x)\nu(\d
x),\quad t,t'\in V.
\]
\end{theorem}

\smallskip

The estimators $\tilde N_n,\tilde{\mathcal{N}}_n$ are `linear' in the
observations $\PP_{\Delta, n}$, and their consistency relies on the assumption
that $\Delta_n$ tends
to zero fast enough, in Theorems~\ref{thm:UCLTnaive} and \ref{thm:UCLTnaive2} at least of order $\Delta_{n}=o(n^{-1/(s+1)})$ for $s\in(0,2]$. In both theorems a weaker assumption than
$\Delta_{n}=o(n^{-1/3})$ cannot be expected in general: In typical situations the function $\PP(X_{\Delta_n}\le t)$, $t<0$, can be
expressed in terms of $\Delta_n$ as a series expansion
\[\PP(X_{\Delta_n}\le t)=\nu((-\infty,t])\Delta_n+b_t\Delta_n^2+O(\Delta_n^3),~~b_t \in (0,\infty).\]
For a compound Poisson process this follows
by conditioning on the number of jumps but it also holds in more general
infinite activity cases \cite[see][]{FLH09}. From the expansion
we see that the approximation error
$\Delta_n^{-1}\PP(X_{\Delta_n}\le t)-\nu((-\infty,t])$ will not decay
faster than $\Delta_n$, and the assumption $\Delta_n=o(1/\sqrt{n\Delta_n})$ is expressed
equivalently as $\Delta_{n}=o(n^{-1/3})$.

\subsection{The spectral estimation approach}

Instead of relying on $\Delta\to0$ one can identify the L\'evy measure by the
L\'evy--Khintchine
formula
\begin{align}\label{eqLevyKhintchine}
  \phi_\Delta(u):=\E[e^{iuX_k}]=e^{\Delta\psi(u)}, \quad
  \psi(u)=-\frac{\sigma^2u^2}{2}+i\gamma u+\int_{\R}\big(e^{iux}-1-iux\big)\nu(\d x),~u \in \R,
\end{align}
which we give here in Kolmogorov's version (valid under \eqref{eqSecMom}, see (8.8) in \cite{sato1999}).
Differentiating the characteristic exponent $\psi(u)=\Delta^{-1}\log\phi_\Delta(u)$, one sees
\begin{align}\label{eqIdent}
  \psi''(u)&=\frac{\phi_\Delta''(u)\phi_\Delta(u)-(\phi_\Delta')^2(u)}{\Delta\phi_\Delta^2(u)}
  =-\sigma^2-\F[x^2\nu](u),
\end{align}
where $\F f(u):=\int e^{iux}f(x)\d x$ and $\F \mu(u) := \int e^{iux}\mu(\d x)$ for any $f\in L^1(\R)\cup L^2(\R)$ and any finite measure $\mu$, respectively, denotes the Fourier transform. If $\F^{-1}$ is the inverse Fourier transform we hence have
\begin{equation}\label{eqident2}
-\mathcal F^{-1}[\psi''] =  \sigma^2 \delta_0 + x^2 \nu.
\end{equation}
In contrast to (\ref{eqWeakConvNu}) this identification of $\nu$ is nonlinear in
$\varphi_\Delta = \F \PP_{\Delta}$, but has the remarkable advantage of being
nonasymptotic and valid for all $\Delta>0$, without relying on a high-frequency
approximation $\Delta \to 0$.  This was exploited in \cite{NicklReiss2012} to
show that a plug-in of the empirical characteristic function $\mathcal F
\PP_{\Delta, n}$ into (\ref{eqIdent}) can result, for a (naturally)
restricted class of L\'evy processes, in efficient recovery of $\mathcal N(t), t
\neq 0,$ \textit{without} the requirement $\Delta \to 0$.
In the low-frequency case only
L\'evy processes without diffusion
component can be covered. Our high-frequency setting allows us to drop
this (otherwise necessary) restriction and to treat L\'evy processes with diffusion
component and with L\'evy measures from a much wider class.

Replacing $\phi_\Delta(u)$ in (\ref{eqIdent}) by the
empirical characteristic function of the observed increments,
\[\phi_{\Delta,n}(u):=\F \PP_{\Delta, n} (u) = \frac{1}{n}\sum_{k=1}^ne^{iuX_k},\]
(and its derivatives $\phi_{\Delta, n}^{(i)}, i=1,2$, respectively), we obtain an empirical plug-in estimate
$\hat\psi_n''$ of $\psi''$. Recalling the definitions of $g_t, f_t$
in Assumption~\ref{assProc}(d) and in Theorem~\ref{thm:UCLTnaive2}, respectively, the
resulting estimators of $N,\mathcal
N$ are given by
\begin{align}
  \hat N_n(t)&:=\int_{\R}g_t(x)\F^{-1}\left[(-\hat\psi''_n-\hat\sigma^2)\mathcal
F K_h \right](x)\d x,\label{eq:kernelEst}\\
\hat{\mathcal{N}}_n(t)&:=\int_{\R}x^{-2}f_t(x)\F^{-1}\left[
(-\hat\psi''_n-\hat\sigma^2)\mathcal
F K_h \right](x)\d x.\notag
\end{align}
Here $K_h$ is a kernel such that $\F K_h$ has compact support, specified in detail below, ensuring in particular that $\hat N_n,\hat{\mathcal{N}}_n$ are well-defined (on sets of probability
approaching one). Moreover,  $\hat\sigma^2$ is any pilot estimate of $\sigma^2$.
 We can estimate $\sigma^2$ for instance as in
\cite{JacodReiss2013} by
\begin{align}\label{eq:sigmahat}
 \hat\sigma^2:=\frac{2}{\Delta u_n^2}\log(|\phi_{\Delta,n}(u_n)|)\qquad \text{
with }u_n:=\sqrt{\frac{2c_0\log(n)}{\Delta \sigma_{\max}^2}},
\end{align}
where $c_0>0$ is a suitable numerical constant, and if we assume a lower bound
on the characteristic function determined by $\sigma_{\max}>0$.
Under suitable conditions Proposition \ref{sigmaest} below entails that the
estimator~$\hat\sigma^2$ satisfies
\begin{align}\label{eq:sigmarate}
 \hat\sigma^2-\sigma^2=o_P((n\Delta)^{-1/2}),
\end{align}
and hence is negligible in the limit process $\mathbb G$ in the next theorem.
While the construction of an optimal estimator of $\sigma$ in the setting
considered here is a topic of independent interest, Theorem~\ref{thmUniform}
below will hold for \textit{any} plug-in estimator that satisfies
(\ref{eq:sigmarate}).

We regularise with a band-limited kernel $K_h:=h^{-1}K(h^{-1}\bull)$ of
bandwidth $h>0$. The following properties of $K$ are supposed:
\begin{gather}\label{propKernel}
  \begin{split}
    \int_{\R} K(x)\d x=1,\qquad \int x^lK(x)\d x=0 \quad\text{ for
}l=1,\dots,p,\\
    \supp \mathcal F K\subset[-1,1],\quad x^{p+1}K(x)\in L^1(\R), ~~p \in \mathbb N.
  \end{split}
\end{gather}
The main result for the spectral estimators is the following theorem,
where $\mathbb{G}$ and $\mathbb{W}$ are tight Gaussian random variables arising
from the same Gaussian processes as in
Theorems~\ref{thm:UCLTnaive} and \ref{thm:UCLTnaive2}, respectively. For
Part~(ii) we recall the definition $f_t=\1_{(-\infty,t]}$ for $t<0$ and
$f_t=\1_{[t,\infty)}$ for $t>0$.
\begin{theorem}\label{thmUniform}
Grant Assumptions~\ref{assProc}(a)-(c) and let $s>0$.
Let the kernel satisfy \eqref{propKernel}
with $p\ge s\vee 2$ and choose
$h_n\sim\Delta_n^{1/2}$.
Let either $\sigma^2$ be known (in which case $\hat \sigma^2 := \sigma^2$), or
let $\hat\sigma^2$ be any estimator satisfying \eqref{eq:sigmarate}.
Suppose $n\to\infty$ and $\Delta_{n}\to0$ such that
\begin{align*}
 n\Delta_{n}\to\infty,\quad \Delta_{n}=o(n^{-1/(s+1)})\quad\text{ and }\quad
\log^4(1/\Delta_n)=o(n\Delta_n).
\end{align*}
\begin{enumerate}
 \item\label{global}
Grant Assumption~\ref{assProc}(d) for $s>0$, for $V=\R$ and
for constants $c_t$ with $\sup_{t\in\R}c_t=c<\infty$.
Then
\[
\sqrt{n\Delta_n}\big(\hat{N}_{n}-N\big)\to^{\mathcal{L}}\mathbb{G}
\quad\text{in}\quad\ell^{\infty}(\R).
\]
 \item\label{away}
Grant Assumption~\ref{assProc}(d) for $s>0$, for $g_t(x)=x^{-2}f_t(x)$,
for $V=(-\infty,-\zeta]\cup[\zeta,\infty)$, $\zeta>0$, and for constants $c_t$
with $\sup_{t\in
V}c_t=c'<\infty$.
Then
\[
\sqrt{n\Delta_n}\big(\hat{\mathcal{N}}_{n}-\mathcal{N}\big)\to^{\mathcal{L}}
\mathbb{W}
\quad\text{in}\quad\ell^{\infty}(V).
\]
\end{enumerate}
\end{theorem}

\subsection{Limit process and statistical applications}

The continuous mapping theorem with the usual sup-norm
$\|\cdot\|_\infty$ combined with Theorems~\ref{thm:UCLTnaive} and \ref{thmUniform} yields
in particular the limit theorems, as $n \to \infty$,
\begin{equation}
\sqrt{n \Delta} \|\tilde N_n -N\|_\infty \to^{\mathcal L} \|\mathbb
G\|_\infty
\quad
\text{ and }
\quad
\sqrt{n \Delta} \|\hat N_n -N\|_\infty \to^{\mathcal L} \|\mathbb
G\|_\infty.
\end{equation}
This can be used to construct Kolmogorov--Smirnov tests for L\'evy measures and
global confidence bands for the function $N$, as we explain now.

For absolutely continuous L\'evy measures $\nu$ the Gaussian random function $(\mathbb G(t):t\in\R)$ can be realised as a version of
\begin{align}
 \mathbb
G(t)&=\mathbb B \left(\int_{-\infty}^t(1\wedge x^4)\nu(\d
x)\right), ~t \in \R,
\end{align}
where $\mathbb B$ is a standard Brownian motion.
An alternative representation is given by $\mathbb
G(t)=\int_{-\infty}^{t}(1\wedge x^2) \nu(x)^{1/2}\d \mathbb B(x)$, where $\mathbb B$ is a
two-sided Brownian motion.
We have
\begin{align*}
 \PP\left(\sup_{s\le t}|\mathbb G(s)|\ge
a\right)
=\PP\left(\left(\int_{-\infty}^t(1\wedge x^4)\nu(\d
x)\right)^{1/2}\max_{s\in[0,1]}|\mathbb B(s)|\ge a\right),~~~ t \in \R \cup
\{\infty\},
\end{align*}
so that quantiles of the distribution of $\|\mathbb G\|_\infty$ can
be calculated. For example a global asymptotic confidence band
for $N$ can be constructed in the setting of Theorems~\ref{thm:UCLTnaive}
and~\ref{thmUniform} by defining
\begin{align*}
\tilde C_n(t)&:=\left[\tilde N_n(t)-\frac{\tilde d q_{\alpha}}{\sqrt {n \Delta}}, \tilde N_n(t)+\frac{\tilde d
q_{\alpha}}{\sqrt {n \Delta}}\right],~~
\hat C_n(t):=\left[\hat N_n(t)-\frac{\hat d q_{\alpha}}{\sqrt {n \Delta}}, \hat N_n(t)+\frac{\hat d q_{\alpha}}{\sqrt {n\Delta}}\right], ~~ t \in \R,
\end{align*}
with consistent estimators
\begin{align*}
\tilde d&:=\left(\frac{1}{n\Delta}\sum_{k=1}^{n}(1\wedge
X_k^4)\right)^{1/2},\\
\hat d &:=
\left(\int_{\R}(x^{-2}\wedge
x^2)\F^{-1}\left[(-\hat\psi''_n-\hat\sigma^2)\mathcal
F K_h \right](x)\d x\right)^{1/2}
\end{align*}
of the standard deviation $(\int_{\R}(1\wedge x^4)\nu(\d x))^{1/2}$,
and with $q_\alpha$ the upper $\alpha$--quantile, $0<\alpha<1$,
of the distribution of $\max_{s\in[0,1]}|\mathbb B(s)|$ (see Example~X.5(c) in
\cite{feller1971} for its well-known formula).
For the confidence band $C_n$ equal to either
$$\tilde C_n := \left\{f: f(t) \in \tilde C_n(t) ~ \forall t \in \R\right\},
~\text{ or }~\hat C_n := \left\{f: f(t) \in \hat C_n(t) ~ \forall t \in
\R\right\},$$
Theorems~\ref{thm:UCLTnaive} and~\ref{thmUniform} imply, under the respective assumptions, that the asymptotic coverage probability of $C_n$ equals
\begin{align*}
 \lim_{n\to\infty}\PP\left(N(t)\in C_n(t) \; \forall
t\in\R\right)=1-\alpha.
\end{align*}
Theorems~\ref{thm:UCLTnaive} and~\ref{thmUniform} allow likewise the
construction of tests: If $H_0$ is a set of L\'evy measures, let $\mathcal D$ be the set of the corresponding cumulative distribution functions
of the form~\eqref{eq:cdf}. We define $T_n=\1\{\mathcal D\cap
C_n=\varnothing\}$ to reject $H_0$ and accept $H_0$ when $T_n=0$. This test has
 asymptotic level $\alpha$:  if $\PP_\theta$ is the law of a L\'evy process from $\theta \in H_0$ then we have
\begin{align*}
 \lim_{n\to\infty}\PP_\theta(T_n\neq0)\le\alpha,
\end{align*}
assuming $H_0$ satisfies the assumptions of Theorem~\ref{thm:UCLTnaive} or
\ref{thmUniform}.

\subsection{Numerical example}

Let us briefly illustrate the finite sample performance of the two estimation approaches and their corresponding  confidence bands. We apply the procedures to two standard examples of pure jump L\'evy processes: a Gamma process and a normal inverse Gaussian (NIG) process. The empirical coverage of the confidence bands reveals the finite sample level of the associated Kolmogorov--Smirnov test and the size of the bands indicates the power of the test.

The Gamma process has infinite, but relatively small jump activity (its Blumenthal-Getoor index equals zero). Its L\'evy measure  is given by the Lebesgue density
$ \nu(x)=\frac{c}{x}e^{-\lambda x}, x>0,$
and we choose $c=30$ and $\lambda=1$ here. The NIG process can be constructed by subordinating a diffusion with volatility $s>0$ and drift $\theta\in\R$ by an inverse Gaussian process with variance $\kappa>0$. The resulting infinite variation process has  Blumenthal-Getoor index equal to one. The NIG process admits an explicit formula for the jump measure and for its law we apply the simulation algorithm from \cite{contTankov2004}, choosing $s=1.5, \theta=0.1$ and $\kappa=0.5$. Both processes satisfy the assumptions of Theorems~\ref{thm:UCLTnaive} and \ref{thmUniform}, cf. Section~\ref{secExamples}.

\begin{figure}[t]\centering
 \includegraphics[width=7.5cm,height=4cm]{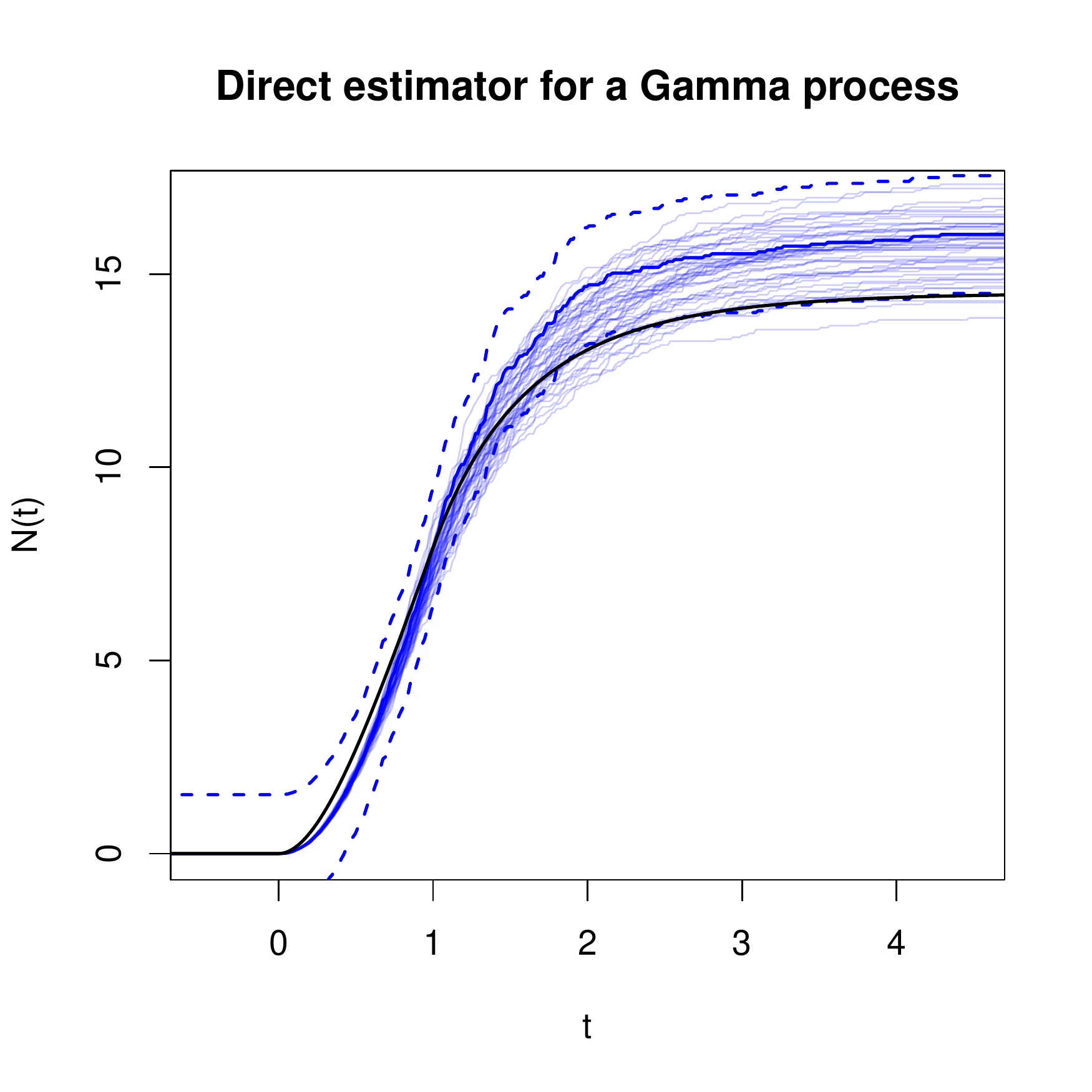}
 \includegraphics[width=7.5cm,height=4cm]{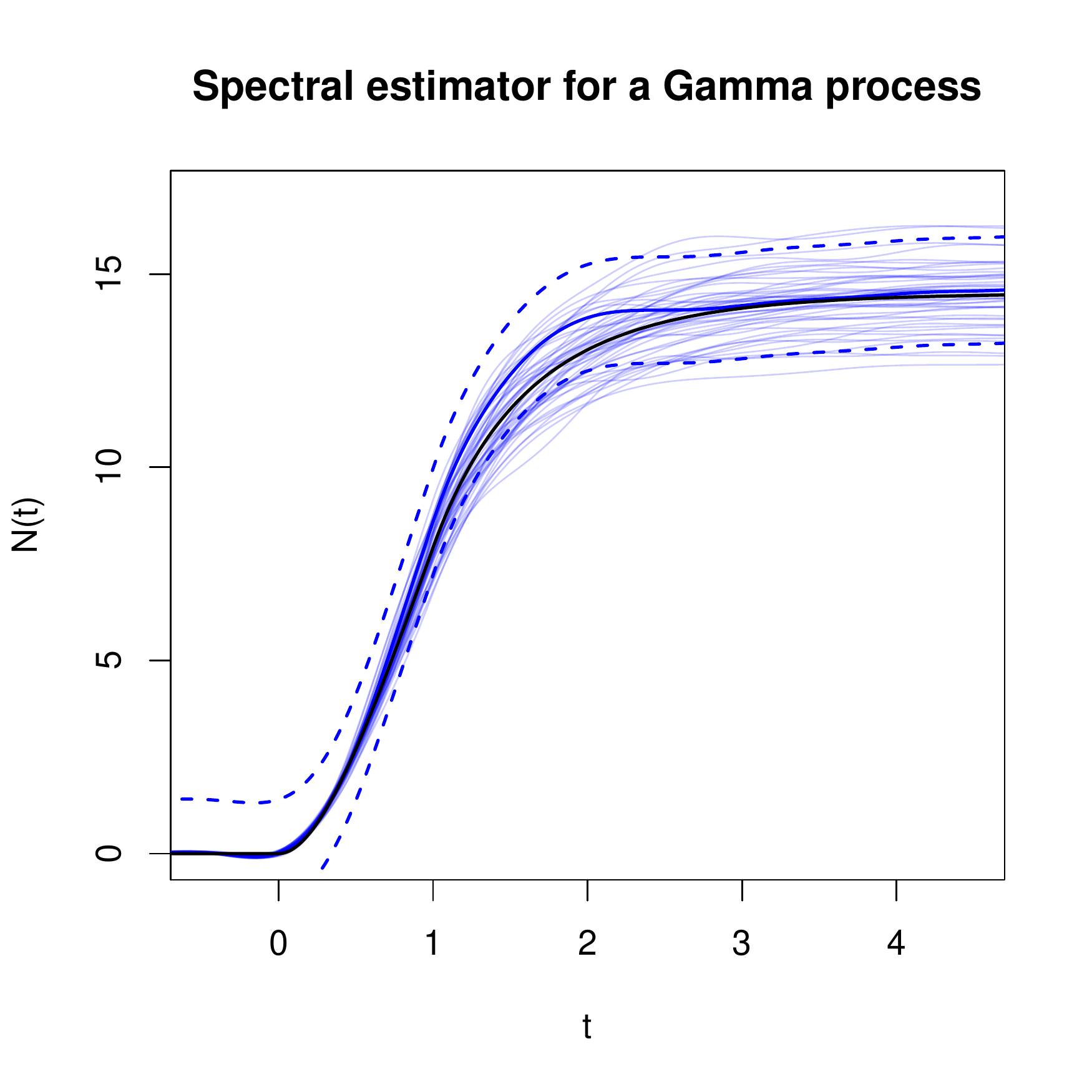}
 \hfill
 \includegraphics[width=7.5cm,height=4cm]{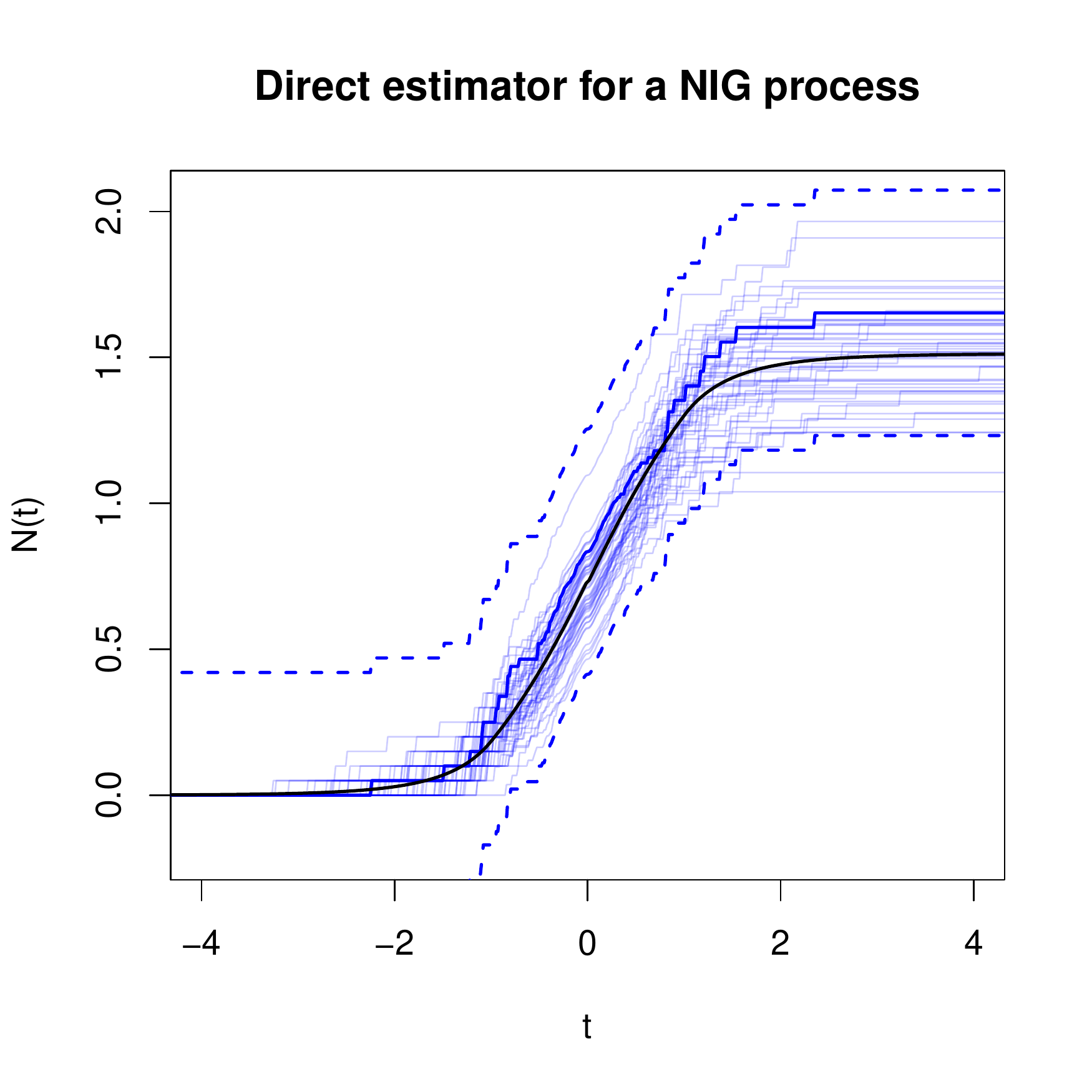}
 \includegraphics[width=7.5cm,height=4cm]{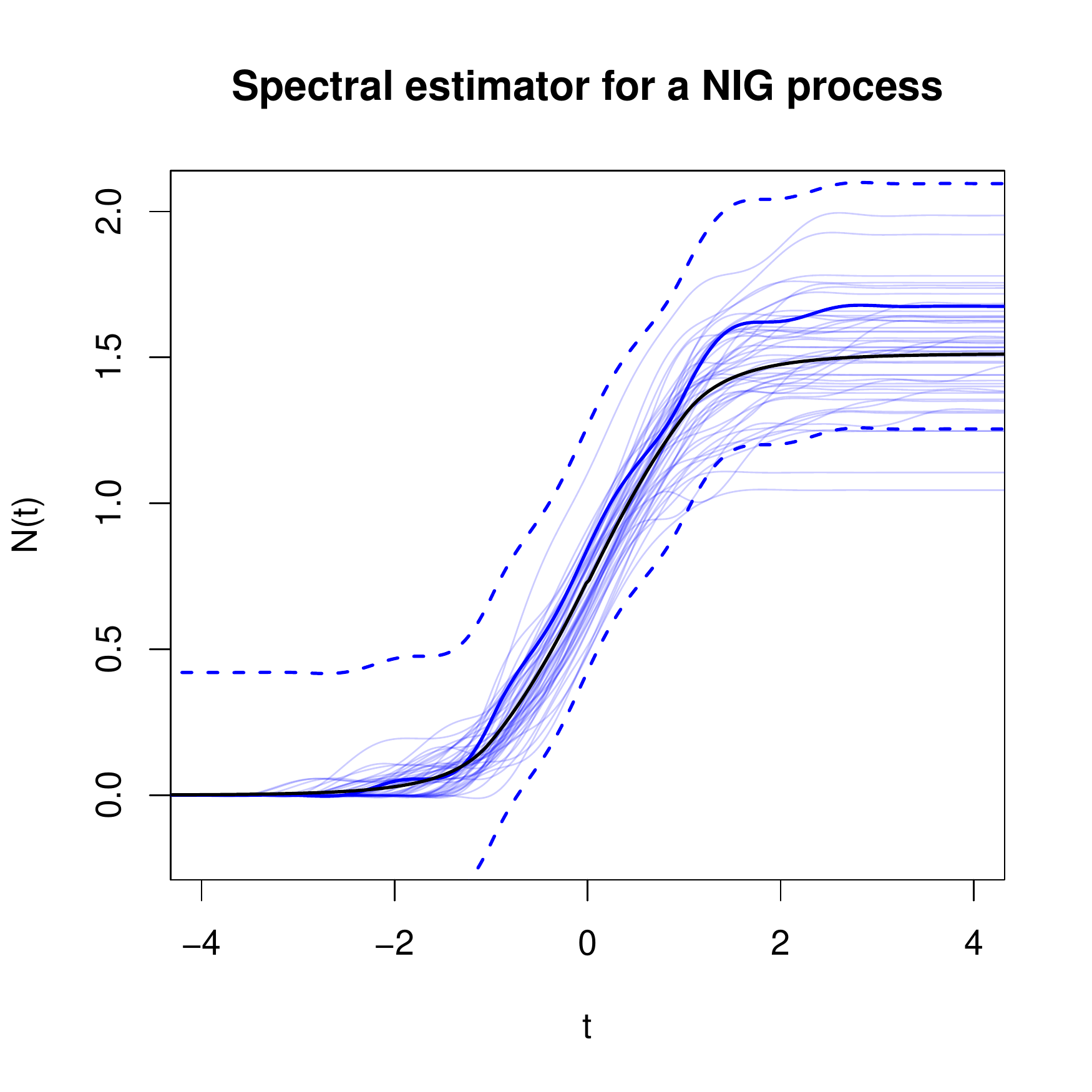}
  \caption{Direct estimator $\tilde N_n$ (\emph{left}) and spectral estimator $\hat N_n$ (\emph{right}) for the Gamma  (\emph{top}) and  NIG process (\emph{bottom}). Each time 50 estimators  (\emph{light blue}) and the true distribution function (\emph{black}) are shown. One estimator (\emph{blue, solid}) with its asymptotic 0.9-confidence band (\emph{blue, dashed}) is highlighted.}\label{fig:sim}
\end{figure}

We simulate $n=2000$ increments with observation distance $\Delta=0.01$. For the spectral estimator we apply a flat top kernel and the universal bandwidth choice $h=\sqrt \Delta$ which turned out to perform well in a variety of settings. Figure~\ref{fig:sim} shows the true distribution-type function $N$, the direct estimator $\tilde N_n$ from \eqref{eq:naiveEst} and the spectral estimator $\hat N_n$ from \eqref{eq:kernelEst} for 50 simulations. In each setting the confidence band for level $\alpha=0.9$, as constructed in the previous section, is plotted for the first simulation result. We clearly see the higher activity of small jumps of the NIG process from the linear growth of  $N$ at zero. On the other hand, the choice of our process parameters yields more pronounced tails of the jump measure for the gamma process.

By construction, the direct estimator is not smooth. For the Gamma process it possesses a significant bias. The intensity of the small jumps is systematically underestimated which results in an overestimation of the larger jumps and thus too large values of $\tilde N_n(t)$ for $t$ large. For the choice $\Delta=0.001$ this bias of the direct estimator is already negligible. In the simulations of the NIG process, $\tilde N_n$ achieves good results that coincide with the asymptotic theory.  In the simulations for the Gamma process the empirical coverage of the $\alpha=0.9$ confidence bands in 500 Monte Carlo iterations is 0.86 for the Gamma process. The direct estimator has an empirical coverage of 0.59, reflecting the bias problem mentioned above. For the NIG process both estimators yield bands covering the true $N$ uniformly in $92\%$ of cases.

\section{Unifying empirical process} \label{secProc}

The key probabilistic challenge in the proofs of Theorems \ref{thm:UCLTnaive} - \ref{thmUniform} is a uniform central limit theorem for certain smoothed empirical processes arising from the sampled increments (\ref{data}). We show in this section how these processes arise naturally for both estimation approaches considered here.

\smallskip

We will consider slightly more general objects than the distribution function $N(t) = \int_{-\infty}^t  (1 \wedge x^2)  \nu(\d x)$ -- the truncation at one in $(1 \wedge x^2)$ is somewhat arbitrary and, in
particular, not smooth. Other truncations such as $x^2/(1+x^2)$, or
variations thereof can be of interest. To accommodate such examples we thus consider recovery of the functionals
\begin{equation}
N_\rho(t) = \int_{-\infty}^t \rho(x) x^2\nu(\d x), ~~ t \in \R,
\end{equation}
where the `clipping function' $\rho$ satisfies the following condition:
\begin{assumption} \label{clip}
The function $\rho$ satisfies $0< \rho(x) \le C(1 \wedge x^{-2})$ for all $x
\in \R$ and some constant $0<C<\infty$. Moreover, $\rho, x\rho$ are Lipschitz
continuous functions of bounded variation (i.e., their weak derivative is equal
to a finite signed measure).
\end{assumption}
This covers the above examples (with either $\rho(x) = 1 \wedge x^{-2}$ or
$\rho(x) = 1/(1+x^2)$). In the definition of the basic estimator
\eqref{eq:naiveEst} and the kernel estimator~\eqref{eq:kernelEst}, we only need to replace
$(1\wedge x^2)\1_{(-\infty,t]}(x)$ by $x^2g_t(x)$ where now
\begin{equation} \label{grho}
g_t(x):=\rho(x)\1_{(-\infty,t]}(x),
\end{equation}
replacing also $g_t$ in Assumption~\ref{assProc}. The covariance of the limit process in Theorems~\ref{thm:UCLTnaive}
and~\ref{thmUniform} then changes to
\[\E[\mathbb{G}(s)\mathbb{G}(t)]=\int_{\R}x^{4}g_s(x)g_t(x)\nu(\d x)\]
and the according representation of $\mathbb G$ in terms of a reparametrised
Brownian motion is
\begin{align}\label{eqBtransformed}
 \mathbb
G(t)=\mathbb B\left(\int_{\R}x^4g_t^2(x)\nu(\d
x)\right)=\mathbb B\left(\int_{\R}x^4\rho^2(x)\1_{(-\infty,t]}(x)\nu(\d
x)\right).
\end{align}

Let us turn to the main purpose of this section: We start with the direct estimator $\tilde N_n$, which is easier to analyse.
The estimation error of $\tilde{N}_{n}$ can be decomposed as follows
\begin{align}
\tilde{N}_{n}(t)-N_{\rho}(t)&=
\int_{\R}x^{2}g_t(x)\big(\Delta^{-1}\PP_{
\Delta,n}
(\mathrm{d}x)-\nu(\mathrm{d}x)\big)\nonumber \\
&=  \int_{\R}x^2g_t(x)
\big(\Delta^{-1}\PP_{\Delta}(\mathrm{d}x)-
\nu(\mathrm{d}x)\big)
+\int_{\R}g_t(x)
\frac{x^{2}}{\Delta}(\PP_{\Delta,n}-\PP_{\Delta})(\d
x)\nonumber \\
&=:  B(t)+S(t),\label{eq:decomp}
\end{align}
for any $t\in\R$. The first term $B$ is a deterministic approximation error
and the rough idea for controlling it is to view $\PP_\Delta$ as an approximate
identity and to use
similar arguments as for the approximation error of a kernel estimator.
The second term $S$ is the main stochastic error term driven by the empirical
process
\begin{equation} \label{naivproc}
 \sqrt{n\Delta}\left(\frac{x^{2}}{\Delta}\PP_{\Delta,n}-\frac{x^{2}}{\Delta}\PP_
{
\Delta }
\right)=\sqrt{\frac{n}{\Delta}}x^{2}\big(\PP_{\Delta,n}-\PP_{\Delta}\big),
\end{equation}
where the scaling follows from the intuitive observation that the $X_k$'s are
drawn i.i.d.~from law $\PP_{\Delta}$ and hence satisfy, using that $\PP_\Delta$
is an infinitely divisible distribution, $$\Var\left(\sum_{k=1}^n X_k\right) =
\Var (L_{n \Delta}) = n \Delta \Var (L_1).$$
Turning our attention to the second estimator we decompose $\hat N_n-N_\rho$
into three error terms, using \eqref{eqIdent}:
\begin{align}
  \hat N_n(t)-N_\rho(t)
  &=
\int_{\R}(g_t(x)\F^{-1}\left[(-\hat\psi''_n-\hat\sigma^2)\mathcal
F K_h \right](x)\d x-x^2 g_t(x)\nu(\d x))\notag\\
  &=\int_{\R} g_t(x)\big(K_h\ast\big(y^2\nu(\d y)\big)-x^2\nu\big)(\d
x)\label{eqErrorDecomp}\\
  &\quad+\int_{\R} g_t(x)\F^{-1}\Big[\mathcal F
K_h(u)\big(\psi''(u)-\hat\psi_n''(u)\big)\Big](x)\d x\notag\\
  &\quad+(\sigma^2-\hat\sigma^2)\int_{\R}g_t(x)K_h(x)\d x.\notag
\end{align}
The first term is a deterministic approximation error, which can be bounded by
Assumption~\ref{assProc}(d) on the smoothness. The last term will be negligible since we assume
that $\hat\sigma^2$ converges to $\sigma^2$ with a faster rate than $1/\sqrt{n\Delta}$.
The key stochastic term is the second one. Compared to the basic estimator
$\tilde N_n$ we face the additional difficulty that $\hat \psi_n''$ depends
nonlinearly on $\PP_{\Delta,n}$. The following result shows that even after
linearisation the resulting term is still different from the basic process
$\tilde N_n -\E \tilde N_n$ in that it performs a division by $\phi_\Delta$ in
the spectral domain.

\begin{proposition}\label{propRem}
Grant Assumptions~\ref{assProc}(a) and assume $\sup_{u\in[-1/h_n,1/h_n]}|\phi_{\Delta_n}(u)|^{-1}\lesssim1$ for some $h_n \to 0, \Delta_n \to 0$. Let the function
$$m(u):=\frac{\mathcal F K(h_nu)}{\phi_{\Delta_n}(u)}$$ satisfy uniformly for
$h_n,\Delta_n\to0$, $\|m\|_\infty\lesssim 1$  (valid for $K$ as in
(\ref{propKernel}) and $h_n \sim \sqrt {\Delta_n}$). If $n\Delta_n\to\infty$ and
$h_n\to 0$ with $h_n\gtrsim \Delta_n^{1/2}$, then we
have
\[\int_{\R} g_t(x)\F^{-1}\Big[\mathcal F
K_{h_n}(u)(\psi-\hat\psi_n)''(u)\Big](x)\d
x=M_{\Delta,n}+o_P(1/\sqrt{n\Delta_n}),
\]
where
\begin{align}
   M_{\Delta,n}&:=-\Delta_n^{-1}\int_{\R} g_t(x){\cal
F}^{-1}[\phi_{\Delta_n}^{-1}(\phi_{\Delta_n,n}''-\phi_{\Delta_n}''){\cal
F}K_{h_n}](x)\d x\notag\\
&\phantom{:}= \int_{\R} g_t(x)
\left(\frac{x^2}{\Delta}(\PP_{\Delta,n}-\PP_{\Delta})\right)*\F^{-1}[m](\d
x).\label{Mproc}
\end{align}
\end{proposition}

We refer to $M_{\Delta,n}$ as the \emph{main
stochastic term}. To accommodate both (\ref{naivproc}) and (\ref{Mproc}) we now study empirical processes
\begin{equation} \label{genproc0}
\sqrt{\frac{n}{\Delta}}( x^2 (\PP_{\Delta, n}-\PP_{\Delta})) \ast \mathcal
F^{-1} m.
\end{equation}
for general  $(n,\Delta)$-dependent Fourier multipliers $m: \R \to \C$ satisfying the following condition.
\begin{assumption} \label{multass}
For every $n, \Delta$ the twice differentiable functions $m = m_{n, \Delta}: \R
\to \C$ are
either such that

(a) $\F^{-1}[m_{n, \Delta}]$, $\F^{-1}[m_{n, \Delta}']$ are
finite signed measures with uniformly bounded total variations, \newline
or such that

(b) $\F^{-1}[m_{n, \Delta}]$ is real-valued
and $m_{n, \Delta}$ is supported in $[-C\Delta^{-1/2}, C \Delta^{-1/2}]$ for
some fixed constant $C>0$.

\medskip

Moreover, letting $\Delta=\Delta_n \to 0$ as $n
\to
\infty$ we assume that $m_{n,\Delta} \to 1$ pointwise on $\R$, that
\[\|(1+|u|)^{k}m_{n, \Delta}^{(k)}\|_\infty\le c,
\quad k\in\{0,1,2\},\]
for some $0<c<\infty$ independent of $n, \Delta$ and that
\[\|m'_{n, \Delta}\|_{L^2}\to 0,\quad \Delta^{-1/2}\|m''_{n,
\Delta}\|_{L^2}\to0.\]
\end{assumption}
The above assumption is an adaptation of the usual Mikhlin-type Fourier
multiplier conditions to the situation
relevant here \citep[see][Cor. 4.11]{girardiWeis2003}.
It ensures that $m$, $m'$ act as norm-continuous Fourier multipliers on
suitable function spaces, which will be a key tool in our
proofs. Obviously Assumption \ref{multass} covers the case $m=1$ relevant in
(\ref{naivproc}) above. Moreover, we show in Proposition~\ref{multver} below
that it
also covers
$m=\mathcal F K_{(\Delta)}/\phi_\Delta$ under our conditions on $\phi_\Delta$
and $K_{(\Delta)}$,
where $K_{(\Delta)}$ denotes a kernel as in (\ref{propKernel}) with bandwidth depending on $\Delta$.
It includes other situations not studied further here, too, such as
smoothed empirical processes
based on $\PP_{\Delta, n}$ convolved with an approximate identity $K_h= h^{-1}
K(\cdot/h), h:=h_n \to 0, \int K=1,$ upon setting $m=\mathcal F K_h$.

With the definition of general $m=m_{n,\Delta}$ at hand we can now
unify the second term $S(t)$ in~\eqref{eq:decomp} and the main stochastic error
\eqref{Mproc}, and study the
smoothed
empirical process
\begin{align}
 \mathbb{G}_n(t)&:= \sqrt{n\Delta}\int_{\R} g_t(x)
\left(\frac{x^2}{\Delta}(\PP_{\Delta,n}-\PP_{\Delta})\right)*\F^{-1}[m](\d
x),\label{genproc}
\\
& = \sqrt{n \Delta}\int_{\R}\F^{-1}[m(-u)\F [g_t](u)](x)
\frac{x^2}{\Delta}(\PP_{\Delta,n}-\PP_\Delta)(\d x), ~t \in \R,\notag
\end{align}
the identity following from Fubini's theorem and standard properties of
Fourier transforms.

When $t$ is a fixed point in $\R$ and $m=1$ one shows without difficulty that, as $n \to \infty$,
$$\Var(\mathbb G_n(t)) \to \int_{\R}  x^4 g_t^2(x) \nu(\d x)$$ whenever
$\nu(\{t\})=0$. More generally one can show
convergence
of the finite-dimensional distributions of the process $(\mathbb G_n(t): t \in
\R)$ to the process $(\mathbb G(t): t \in
\R)$ from Theorem~\ref{thm:UCLTnaive}.
\begin{proposition} \label{fidi0}
Let $\Delta= \Delta_n \to 0$ in such a way that $n \Delta_n \to \infty$.
Suppose the L\'evy process satisfies Assumption~\ref{assProc}, that $\rho$ satisfies Assumption \ref{clip}, and that $m$ satisfies Assumption \ref{multass}.
Then as $n \to \infty$ we have, for any $t_1, \dots, t_k \in \R$, that
   \[
      \left[\mathbb G_n(t_1), \dots, \mathbb G_n(t_k) \right]
\rightarrow^{\mathcal L} \left[\mathbb G(t_1), \dots \mathbb G(t_k)\right].
   \]
 \end{proposition}
We remark that in this proposition we can omit $(1\wedge x^4)\nu\in
\ell^\infty(\R)$ from Assumption~\ref{assProc} as it is only needed
later in the proof of the tightness of the process $\mathbb{G}_n$.

By sample-continuity of Brownian motion, and
since the integral in \eqref{eqBtransformed} takes values in a fixed compact
set, we deduce that there exists a version of $(\mathbb G(t):t\in\R)$ with
uniformly
continuous sample paths for the intrinsic covariance metric
\[d^2(s,t)=\int_{\R}  x^4(g_t(x)-g_s(x))^2\nu(\d
x)=\int_{\R}  x^4\rho^2(x)\1_{(s\wedge t,s\vee t]}\nu(\d
x),\]
and that, moreover, $\R$ is totally bounded with respect to $d$. As a consequence we
obtain:
\begin{lemma} \label{levybrown}
Grant Assumption~\ref{clip}. For $g_t$ as in (\ref{grho}) and any L\'evy measure $\nu$, the law of the
centred Gaussian process
$\{\mathbb{G}(t):t\in\R\}$
with covariance
\[
\E[\mathbb{G}(t)\mathbb{G}(t')]=\int_{\R}x^4 g_t(x) g_{t'}(x) \nu(\d
x),\quad t,t'\in\R
\]
defines a tight Gaussian Borel random variable in $\ell^\infty(\R)$. In particular, there exists a
version of the process $(\mathbb G(t): t \in \R)$ such that $$\sup_{t \in \R} |\mathbb G(t)|<\infty ~a.s.$$
\end{lemma}

The most difficult part in the proofs of Theorem~\ref{thm:UCLTnaive}
and~\ref{thmUniform} is to show
that $\mathbb G_n$ converges in law  to $\mathbb G$ in the space
$\ell^\infty(\R)$ of bounded functions on the real line. Given that convergence
of the finite-dimensional distributions and tightness of the limit process have
already been established, this can be reduced to showing asymptotic
equicontinuity of the process $(\mathbb G_n(t): t \in \R)$, or equivalently,
uniform tightness of the random variables $\mathbb G_n$ in the Banach space
$\ell^\infty(\R)$ (see Section 1.5 in \cite{vanderVaartWellner1996} and
(\ref{asymptotic_equicontinuity}) below for precise definitions).

\begin{theorem} \label{main1}
 Let $\Delta= \Delta_n \to 0$ in such a way that $n \Delta_n \to \infty$ and
$\log^4(1/\Delta_n)=o(n\Delta_n)$.
Suppose the L\'evy process satisfies Assumption~\ref{assProc}, that $\rho$ satisfies Assumption \ref{clip}, and that $m$ satisfies Assumption \ref{multass}. Then the
process $\mathbb G_n$ from~(\ref{genproc}) is asymptotically equicontinuous in
$\ell^\infty(\R)$. In particular, $\mathbb G_n$ is uniformly tight in
$\ell^\infty(\R)$,
\begin{equation*}
\mathbb G_n \to^{\mathcal L} \mathbb G \quad\text{in} \quad\ell^\infty(\R),
\end{equation*}
and
\begin{equation}
\sup_{t \in \R} \left|\int_{\R}g_t(x)
\left(\frac{x^2}{\Delta}(\PP_{\Delta,n}-\PP_{\Delta})\right)*\F^{-1}[m](\d x)
\right| = O_P\left(\frac{1}{\sqrt {n \Delta}} \right).
\end{equation}
\end{theorem}

\smallskip

The proof is based on ideas from the theory of
smoothed empirical processes \citep[in particular from][]{GineNickl2008}.
The main mathematical
challenges consist in dealing with
envelopes of the empirical process that can be as large as  $1/\sqrt
\Delta\to\infty$ in the high-frequency setting,
and in accommodating the presence of an $n$-dependent Fourier multiplier $m$
that needs to be general enough to allow for $m=\mathcal F K_{h}/\phi_\Delta$. The latter
requires the treatment of empirical processes that cannot be controlled with the
standard bracketing or uniform metric entropy techniques.
Our proofs rely on direct arguments for symmetrised empirical processes
inspired by \cite{GineZinn1984}
and on sharp bounds on certain covering numbers based on a suitable
Fourier integral operator inequality for $\F^{-1}[m]$ in
$L^2(\PP_\Delta)$-norms.

\section{Discussion and examples} \label{discussion}

\subsection{Regularity of $x^2 \nu$ and the Blumenthal--Getoor
index}\label{sec:discussion}

The regularity index $s>0$ in Assumption \ref{assProc} measures the smoothness of the function $g_t(-\cdot)\ast(x^2\nu)$. When $x^2\nu$ is sufficiently regular away from the origin, this will equivalently measure the
smoothness of the function $t \mapsto \int_{-\infty}^t x^2\nu(\d x)$, and hence is effectively driven by the singularity that $\nu$ possesses at zero. The latter can be quantitatively measured by the
\cite{blumenthalGetoor1961}-index
\begin{align}
\beta: & =\inf\left\{\alpha>0:\int_{|x|<1}|x|^{\alpha}\nu(\d
x)<\infty\right\}=\inf\left\{\alpha>0:\lim_{r\downarrow0}r^{\alpha}\int_{
\R\setminus[-r,r]}\d\nu=0\right\}.\label{eq:blumenthalGetoor}
\end{align}
The Blumenthal--Getoor index $\beta$ takes values in $[0,2]$ and we have
$\int_{\R}|x|^\alpha\nu(\d x)=c<\infty$ for all $\alpha\in(\beta,2]$ ($\alpha=2$ if $\beta=2$).
In fact, for such $\alpha$ and for all intervals $[a,b]$ containing the origin
\begin{align*}
 \int_a^b |x|^2\nu(\d x) \le \int_a^b|x|^\alpha\nu(\d x) (b-a)^{2-\alpha}\le c
(b-a)^{2-\alpha}.
\end{align*}
Provided $\nu$ is smooth away from zero this shows that the H\"older smoothness
of $\int_{-\infty}^t x^2\nu(\d x)$ is at least $2-\beta^+$, where
$\beta^+>\beta$ and $\beta^+\ge1$.
For a singularity of the from $\nu(x)=|x|^{-\beta-1}$, $\beta\in(1,2)$, which
corresponds to Blumenthal--Getoor index $\beta$, we have $\int_{|x|<t}x^2\nu(\d
x)=2(2-\beta)^{-1}t^{2-\beta}$ showing that the H\"older smoothness is at most
$(2-\beta)$.
This argument can be extended
to the case where the symmetrised L\'evy density  $\tilde\nu$ is
regularly
varying:
If $\tilde\nu$ is regularly varying with exponent
$-(\beta+1)$ at zero then
$\int_{|x|<t}|x|^2\nu(\d x)$ is regularly varying of exponent $(2-\beta)$
at zero by a Tauberian theorem \citep[see e.g.][Thm. VIII.9.1]{feller1971}.
For Blumenthal--Getoor index $\beta\in(1,2]$ this means that the H\"older
regularity of $\int_{-\infty}^t x^2\nu(\d x)$ is at most $(2-\beta)$.

\subsection{The drift parameter $\gamma$}

None of the above estimators $\tilde N, \hat N, \tilde {\mathcal N}, \hat {\mathcal N}$ require knowledge, or estimation, of the drift parameter $\gamma$, which, at any rate, can be naturally estimated by $L_{n\Delta}/(n\Delta)$. It is interesting to note that the `nonlinear' estimator $\hat N_n$ is even invariant under a change of the drift parameter $\gamma$, as the following lemma shows.

\begin{lemma}\label{lemDrift}
Let $Y_k:=X_k-\Delta\gamma,k=1,\dots,n,$ which are increments of a L\'evy
process with characteristic triplet $(\sigma,0,\nu)$. Denoting the estimators
\eqref{eq:kernelEst} based on $(X_k)$ and $(Y_k)$ as $\hat N_{X,n}$ and $\hat
N_{Y,n}$, respectively, we obtain
 \[
    \forall t\in\R: \hat N_{X,n}(t)=\hat N_{Y,n}(t).
 \]
\end{lemma}
\begin{proof}
  The drift causes a factor $e^{-i\Delta\gamma u}$ in the empirical characteristic function $\phi_{\Delta,n,Y}$ such that
  \[ \hat\psi_{n,Y}''(u)=\Delta^{-1} (\log(\phi_{\Delta,n,X}(u))-i\Delta\gamma u)''=\hat\psi_{n,X}''(u).
  \]
  $\hat N_n$ only depends via $\hat\psi_{n}''$ on the observations.
\end{proof}

Consequently, without loss of generality
a specific value of $\gamma$ can be assumed in the proofs
for
the estimator~$\hat N_n$ based on the L\'evy--Khintchine representation. In
particular, the conditions on $\PP_\Delta$ need to be verified only for one
$\gamma$.

\subsection{A pilot estimate of the diffusion coefficient $\sigma$}

\begin{proposition} \label{sigmaest}
Suppose the L\'evy measure satisfies $\int \abs{x}^\alpha\nu(dx)<\infty$ for some $\alpha\in[0,2]$ and the characteristic function is bounded from below via
\[ \abs{\phi_\Delta(u)}\ge \exp(-\Delta\sigma_{max}^2u^2/2)\text{ for all } u\ge 0.\]
Let $\hat \sigma^2$ be as in (\ref{eq:sigmahat}). Then we have, for $c_0$ small enough, as $n \to\infty$, and uniformly in $\Delta\le 1$,
\[\abs{\hat\sigma^2-\sigma^2}=O_P\Big((\log
n)^{(\alpha-2)/2}\Delta^{1-\alpha/2}+(\log n)^{-1}n^{c_0-1/2}\Big).\]
\end{proposition}

The proof follows along the lines of \cite{JacodReiss2013} and is omitted.
The previous discussion and the examples in Section~\ref{secExamples}
below show that the natural connection between smoothness $s$ and
Blumenthal--Getoor index $\beta$ is given by $s=2-\beta$. For such $s$ and with
the choice $c_0=1/6$ the conditions of Theorem~\ref{thmUniform} ensure that
(\ref{eq:sigmarate}) is satisfied provided the infimum in the definition of the
Blumenthal--Getoor index is attained. Otherwise it suffices to replace the
condition $\Delta_n=o(n^{-1/(s+1)})$ by the slightly stronger condition
$\Delta_n=o(n^{-1/(s^-+1)})$ for some $s^-<s$ in order to guarantee
(\ref{eq:sigmarate}).
Other estimators, based for instance on the truncated quadratic variations of
the process, can be considered, and different sets of conditions are possible.
As this is beyond the scope of the present paper, we refer to
\cite{JacodReiss2013} for discussion and references.

\subsection{Bounding $\|x^3\PP_\Delta\|_\infty$}\label{sec:Boundx3P}

A key condition in all results above is a uniform bound on
$\|x^3\PP_{\Delta}\|_\infty$  of order $\Delta$. The following proposition shows
that this condition follows already from $\|x\PP_\Delta\|_\infty\lesssim1$ and
$x^3\nu\in \ell^\infty(\R)$.
We recall that we always assume $\int_{\R}x^2\nu(\d x)<\infty$.
\begin{proposition}\label{Propx3PDelta}
  For any L\'evy process $(L_t:t\ge 0)$ with $\norm{x\PP_\Delta}_\infty\lesssim 1$ and $x^3\nu\in \ell^\infty(\R)$ we have $\|x^3\PP_\Delta\|_\infty\lesssim\Delta$ (with constants uniform in $\Delta$).
\end{proposition}

\begin{proof}
  From $\phi_\Delta''=(\Delta
  \psi''+(\Delta\psi')^2)\phi_\Delta=\Delta\psi''\phi_\Delta+(\Delta\psi'\phi_{
  \Delta/2})^2$ by the infinite divisibility, we conclude
  \begin{equation}\label{eqX2P}
    x^2\PP_\Delta=\Delta \nu_\sigma \ast \PP_\Delta+4(x\PP_{\Delta/2})\ast(x\PP_{\Delta/2}),
  \end{equation}
  where $\nu_\sigma=\sigma^2\delta_0+x^2\nu$. Using $x(P\ast Q)=(xP)\ast Q+P\ast(xQ)$, we infer further
  \[ x^3\PP_\Delta=\Delta\Big( (x\nu_\sigma)\ast \PP_\Delta+\nu_\sigma\ast(x\PP_\Delta)\Big)+8(x\PP_{\Delta/2})\ast (x^2\PP_{\Delta/2}).
  \]
  By assumption and properties of L\'evy processes, we have
$\norm{x\nu_\sigma}_\infty<\infty$, $\PP_\Delta(\R)=1$,
$\nu_\sigma(\R)<\infty$ and $\norm{x^2\PP_\Delta}_{L^1}\lesssim\Delta$.
  This yields
  \[ \norm{x^3\PP_\Delta}_\infty \lesssim \Delta(1+\norm{x\PP_{\Delta}}_\infty + \norm{x\PP_{\Delta/2}}_\infty)\lesssim\Delta.\qedhere
  \]
\end{proof}

The condition $\|x\PP_{\Delta_n}\|_\infty\lesssim 1$ is satisfied for all basic
examples of L\'evy processes like Brownian motion, compound Poisson, Gamma and
symmetric (tempered) $\alpha$-stable processes. For the latter
processes it is interesting to compare the resulting bounds to the small time
estimates by \cite{Picard1997}. The conjecture
that the bound $\norm{x\PP_\Delta}_\infty\lesssim 1$ is universal for arbitrary
jump behaviour near zero, however, is wrong as the case of a completely
asymmetric (tempered) 1-stable process shows where
$\PP_\Delta(-\Delta\log(1/\Delta))\thicksim \Delta^{-1}$ holds, see the
exceptional case in Example~4.5 of \cite{Picard1997}.

\medskip

If $\int_{\R}|x|\nu(\d x)<\infty$ we can define the drift parameter $\gamma_0 :=
\gamma - \int x \nu(\d x)$.
\begin{assumption}\label{AssPDelta}
 Let $(\sigma^2,\gamma,\nu)$ be a L\'evy  triplet and $\nu^+=\nu{\1}_{\R^+}$, $\nu^-=\nu{\1}_{\R^-}$. Consider the following conditions for the two triplets $(\sigma^2,\gamma,\nu^\pm)$:
 \begin{enumerate}
 \item (diffusive case) $\sigma>0$
 \item (small intensity case) $\sigma=0$, $\gamma_0=0$,
$\norm{x\nu^\pm}_\infty<\infty$
 \item (finite variation case) $\sigma=0$, $\gamma_0=0$, $x\nu^\pm$ admits a
Lebesgue density in $\ell^\infty(\R\setminus[-\eps,\eps])$ for all $\eps>0$,
$\eps^{-1}\int_{-\eps}^\eps \abs{x}\nu^\pm(\d
x)\lesssim \eps\nu(\pm\eps)$ for $\eps\downarrow 0$ and
 \[ \liminf_{\eps\downarrow 0}\inf_{t\in(0,1]}
\frac{(t\eps)^{-1}\int_{\abs{x}\le t\eps} x^2\nu^\pm(\d x)} {\eps^2\nu(\pm
\eps)t\log(t^{-1})}>0
 \]
 \item (infinite variation case) $\sigma=0$, $\nu^{\pm}$~admits a Lebesgue
density,
\begin{align*}
 \frac1{\eps}\int_{-\eps}^\eps
x^2\nu^\pm(\d x)\gtrsim \int_{\abs{x}>\eps}\abs{x}\nu^\pm(\d x)+1\quad\text{ for
}
\eps\in(0,1)
\end{align*}
 \end{enumerate}
\end{assumption}

\begin{proposition}\label{Propx1PDelta}
If each of the triplets $(\sigma^2,\gamma,\nu^\pm)$ of the L\'evy process
satisfies one of the Assumptions \ref{AssPDelta}(i)-(iv), then
$\norm{x\PP_\Delta}_\infty\lesssim 1$ holds uniformly in $\Delta$.
\end{proposition}

\subsection{Examples}\label{secExamples}

Let us discuss the applicability of Proposition~\ref{Propx1PDelta} together with the smoothness conditions on the jump measure from Theorem \ref{thmUniform} in a few examples.

 \begin{enumerate}
  \item {\it Diffusion plus compound Poisson process.}\\
  Let $\nu$ be a finite measure on $\R$ with a Lebesgue density. Suppose
$\int_{\R}|x|^{4+\eps}\nu(\d x)<\infty$ for some $\eps>0$ and
$\|x^3\nu\|_\infty<\infty$. Proposition~\ref{Propx1PDelta} yields
$\|x\PP_\Delta\|_\infty\lesssim1$ if either $\gamma_0=0$ and
$\nu(x)\lesssim|x|^{-1}$ as $x\to0$, or if $\sigma>0$.

 For $x^2\nu\in \ell^\infty(\R)$ the global H\"older regularity in Assumption
\ref{assProc}(d) is $s=1$, and for smooth compounding measure $x^2\nu\in
C^r(\R)$ it is satisfied with $s=r+1$.

  \item {\it Self-decomposable L\'evy process.}\\
  The jump measures of self-decomposable L\'evy processes are characterised by
$\nu(\d x)=\frac{k(x)}{|x|}\d x$ for a function $k:\R\to\R_+$ which is
monotonically increasing on the negative half line and decreasing on the
positive one. An explicit example is given by the Gamma process where
$k(x)=ce^{-\lambda x}\1_{\R_+}(x)$ for $c,\lambda>0$. Note that nontrivial
self-decomposable processes have an infinite jump activity. If $k$ is a bounded
function, then Assumption~\ref{AssPDelta}(ii) is fulfilled. The smoothness is
determined by the H\"older regularity of $|x|k(x)$, for instance, Gamma
processes induce regularity $s=2$ at $t=0$ and $C^\infty$ away from the origin.

  \item {\it Tempered stable L\'evy process.}\\
  Let $L$ be a tempered stable process, that is a pure jump process with L\'evy measure given by the Lebesgue density
  \[
    \nu(x)=|x|^{-1-\alpha}\left(c_-e^{-\lambda_-|x|}\1_{(-\infty,0)}(x)+c_+e^{-\lambda_+|x|}\1_{(0,\infty)}(x)\right)
  \]
  with parameters $c_\pm\ge 0,\lambda_\pm>0$ and stability index $\alpha\in(0,2)$. By the exponential tails of~$\nu$ the moment assumptions are satisfied. For the finite variation case $\alpha\in(0,1)$ Assumption~\ref{AssPDelta}(iii) can be verified since $\eps^{-1}\int_{-\eps}^\eps|x|\nu^\pm(x)\d x\sim\eps^{-\alpha}\sim \eps\nu(\pm\eps)$ and the second condition simplifies to $t^{-\alpha}/\log(t^{-1})>0$. In the infinite variation case $\alpha\in(1,2)$ Assumption \ref{AssPDelta}(iv) is satisfied owing to $$\eps^{-1}\int_{-\eps}^\eps x^2\nu^\pm(x)\d x\sim\eps^{1-\alpha}\sim\int_{|x|>\eps}|x|\nu^\pm(x)\d x, ~~\eps\in(0,1).$$

  Outside of a neighbourhood of zero the L\'evy measure is arbitrarily smooth. Due to the cusp of $x^2\nu(x)$ at the origin the global H\"older regularity is in general given by $s=2-\alpha$. In the case $\alpha=1$ and $c_+=c_-$, $x^2\nu$ is already Lipschitz continuous at zero and so $s=2$.

  \item {\it Jump densities regularly varying at zero.}\\
  The first condition in Assumption~\ref{AssPDelta}(iii) holds for regularly varying $\nu$ with $\alpha<1$, that is $\nu^\pm(x)=|x|^{-1-\alpha}l(x)$ with slowly varying $l$ at zero,  by a classical Tauberian theorem \citep[see e.g.][Thm. VIII.9.1]{feller1971}. The second condition then reduces to
  \[ l(t\eps)\ge C_\alpha t^\alpha\log(t^{-1})l(\eps),\quad C_\alpha>0,\]
  uniformly over $t\in(0,1]$ for small $\eps>0$, which is always satisfied for $\alpha>0$. Similarly, Assumption \ref{AssPDelta}(iv) is satisfied  if $\nu^\pm(x)=|x|^{-1-\alpha}l(x)$ holds with $\alpha\in(1,2)$ and a slowly varying function $l$ at zero.
\end{enumerate}

\section{Proofs}

We collect the proofs for Sections~\ref{secEstimators}, \ref{secProc}, \ref{discussion}. Theorem~\ref{main1} is proved in the next section.

\subsection{Proof of Lemma \ref{weakcon}}
The result is a standard -- for convenience of the reader we include a short proof. Using
the L\'evy--Khintchine formula (\ref{eqLevyKhintchine}) we see
\begin{equation} \label{var}
  c_\Delta:=\Delta^{-1}\E[X_1^2]= -\Delta^{-1}\phi_{\Delta}''(0) = -\psi''(0)- \Delta (\psi'(0))^2 \to\sigma^2+\int x^2 \nu(\d x)=:c
\end{equation}
as $\Delta\to0$. The characteristic function of the probability measure
$c_\Delta^{-1}x^2\Delta^{-1}\PP_\Delta(\d x)$ converges pointwise to the
characteristic function of $c^{-1}(\sigma^2 \delta_0+ x^2\nu(\d x))$ as
$\Delta\to0$ since
\begin{align*}
 \frac{1}{c_\Delta \Delta}\int e^{iux}x^2 \PP_\Delta(\d x)=-\frac{1}{c_\Delta
\Delta}\phi_\Delta''(u)&=-\frac{1}{c_\Delta \Delta}(\Delta
\psi''(u)+\Delta^2(\psi'(u))^2)e^{\Delta\psi(u)}\\
&\to \frac{1}{c}(\sigma^2+\int
e^{iux}x^2 \nu(\d x)).
\end{align*}
Therefore, we obtain \eqref{eqWeakConvNu} from L\'evy's continuity theorem.

\subsection{Proof of Theorem \ref{thm:UCLTnaive}}

Using decomposition (\ref{eq:decomp}), Theorem \ref{thm:UCLTnaive} follows from Theorem \ref{main1} with $m=1$ (which trivially satisfies Assumption \ref{multass}(a)), if we can show that the `bias' term $B(t)$ is asymptotically negligible uniformly in $t \in \R$. This is achieved in the following proposition.

\begin{prop}
\label{prop:bias} Grant the assumptions of Theorem~\ref{thm:UCLTnaive}. Then it
holds
\[
\sup_{t\in\R}|B(t)|=\sup_{t\in\R}\Big|\int
g_{t}(x)\big(\Delta^{-1}x^{2}\PP_{\Delta}(\mathrm{d}x)-x^{2}
\nu(\mathrm{d}x)\big)\Big|=\mathcal{O}(\Delta^{s/2}).
\]
\end{prop}
\begin{proof}
We decompose the bias into
\begin{align}
B(t)&=  \int
g_{t}(x)\big(\Delta^{-1}x^{2}\PP_{\Delta}(\mathrm{d}x)-((x^2\nu)\ast\PP_{
\Delta})(\d x)\big)\nonumber \\
 & \qquad+\int g_{t}(x)\big(((x^{2}\nu)\ast\PP_{\Delta})(\d x)-x^{2}\nu(\d
x)\big)\nonumber \\
&=:  B_{1}(t)+B_{2}(t).\label{eq:BiasDecomp}
\end{align}

We start with the first term $B_1(t)$: Using $-(\F f)''=\F[x^{2}f]$ for any
function $f$ satisfying $(1\vee x^{2})f\in
L^{1}(\R)$,
we have
\[
-\F\big[x^{2}\PP_{\Delta}(\mathrm{d}x)]=\phi_{\Delta}
''=(\Delta\psi''+(\Delta\psi')^{2})\phi_{\Delta}.
\]
Plancherel's identity and $\psi''=-\F[x^2\nu]$ then gives
\begin{align*}
B_{1}(t)&=  \frac{1}{2\pi}\int\F
g_{t}(-u)\F\big[\Delta^{-1}x^{2}\PP_{\Delta}(\mathrm{d}x)-((x^2\nu)\ast\PP_{
\Delta})(\d x)\big](u)\,\d u\\
&=  \frac{1}{2\pi}\int\F
g_{t}(-u)\big(-\Delta^{-1}\phi_{\Delta}''(u)-\F[x^2\nu](u)\phi_{\Delta}
(u)\big)\,\d u.\\
&=  -\frac{\Delta}{2\pi}\int\F g_{t}(-u)\psi'(u)^{2}\phi_{\Delta}(u)\,\d u.
\end{align*}
The proofs below will imply that the last integral exists, which in particular justifies the preceding manipulations. We shall repeatedly use that $\sup_{t}\|g_{t}\|_{L^{1}}\le \|\rho\|_{L^{1}}$ and
$\sup_{t}\|g_{t}\|_{BV}\le\|\rho\|_{BV}$ imply
 \begin{equation} \label{boringbound}
 |\F g_{t}(u)|\lesssim(1+|u|)^{-1},u\in\R,
 \end{equation}
uniformly in $t \in \R$. In case a) we can use (\ref{boringbound}),
$\|\phi_\Delta\|_\infty =1$, $|\F[x\nu](u)| \lesssim(1+|u|)^{-1}$, the
hypothesis $\gamma_0=0$ and the resulting identity
$$\psi'(u) = i \F[x\nu](u) - i \left(\int x \nu(\d x) - \gamma \right) = i \F[x \nu](u)$$
to bound
$$|B_1(t)| \le \frac{\Delta}{2\pi}\int |\F g_{t}(-u)||\F[x \nu](u)|^{2}|\phi_{\Delta}(u)|\,\d u  \lesssim \Delta \int \frac{1}{(1+|u|)^3} \d u =O(\Delta).$$
For case b) we will show that
\begin{equation}
\sup_{t\in\R}|B_{1}(t)|\lesssim\Delta^{p}\quad\text{for any}\quad
p\in\big(0,\frac{2-\beta}{\beta}\vee1\big).\label{eq:b1}
\end{equation}
By assumption $\tilde\nu(x)=\nu^+(x)+\nu^-(-x)$ is regularly varying at zero
with
exponent $-(\beta+1)$ and so the function $H(r):=\int_{|x|<r}x^2\nu(x)\d
x=\int_0^r x^2 \tilde\nu(x)\d x$ is regularly varying with exponent $(2-\beta)$
by a Tauberian theorem \citep[Thm. VIII.9.1]{feller1971}. Especially we can
bound $H(r)$ from below, more precisely for any $\beta^-\in(0,\beta)$ there
exists $r_0>0$ such that $H(r)\gtrsim r^{2-\beta^-}$ for all $r\in(0,r_0)$.
By~\citet{orey1968}
there is a constant $c>0$ such that
\begin{equation}\label{eq:phibound}
 |\phi_{\Delta}(u)|\lesssim\exp(-c\Delta|u|^{\beta^{-}})
\end{equation}
for $|u|$ sufficiently large.
On the other hand, it is easily seen that
\begin{equation}\label{eq:psibound}
 |\psi'(u)|\lesssim1+|u|^{(\beta^{+}-1)\vee0}
\end{equation}
for any $\beta^{+}\in(\beta,2)$ and that $|\psi''(u)|$ is bounded.
Especially we have
$\phi_{\Delta},\phi_{\Delta}''\in L^{2}(\R)$.
Collecting the above and using (\ref{boringbound}) implies
\[
\sup_{t\in\R}|B_{1}(t)|\lesssim\Delta\int(1+|u|)^{-1}|\psi'(u)|^{2}|\phi_{\Delta
}(u)|\,\d u.
\]
Let us distinguish the cases $\beta\ge1$
and $\beta<1$, which will yield together (\ref{eq:b1}). We will be using the
bounds for $\phi_\Delta$ and $\psi'$ in \eqref{eq:phibound} and
\eqref{eq:psibound}, respectively.
\begin{enumerate}
\item For $\beta\ge1$ substituting $u=\Delta^{-1/\beta^{-}}z$ yields
\begin{align*}
\sup_{t\in\R}|B_{1}(t)|&\lesssim
\Delta\int(1+|u|)^{(2\beta^{+}-3)}\exp(-c\Delta|u|^{\beta^{-}})\,\d u\\
&\lesssim
\Delta^{(\beta^{-}-2\beta^{+}+2)/\beta^{-}}\int((1+|z|)^{2\beta^{+}-3}\vee|z|^{
2\beta^{+}-3})\exp(-c|z|^{\beta^{-}})\,\d z,
\end{align*}
where the integral in the last display is finite owing to $\beta^{+}>1$.
Noting that $2\beta^{+}-\beta^{-}>\beta$, we conclude that
$|B_{1}|\lesssim\Delta^{p}$
for any $p<(2-\beta)/\beta$.
\item For $0<\beta<1$ boundedness of $|\psi'|$ and the same substitution
yields for any $\delta>0$
\begin{align*}
\sup_{t\in\R}|B_{1}(t)|&\lesssim
\Delta\int(1+|u|)^{-1}\exp(-c\Delta|u|^{\beta^{-}})\,\d u\\
&\le
\Delta^{1-1/\beta^{-}}\int(1+\Delta^{-1/\beta^{-}}|z|)^{-1+\delta}\exp(-c|z|^{
\beta^{-}})\,\d z\\
&\le  \Delta^{1-\delta/\beta^{-}}\int|z|^{-1+\delta}\exp(-c|z|^{\beta^{-}})\,\d
z.
\end{align*}
By choosing $\delta$ sufficiently small, we obtain $|B_{1}|\lesssim\Delta^{p}$
for any $p<1$.
\end{enumerate}
Let us now consider $B_{2}$ in (\ref{eq:BiasDecomp}) which we can
write as
\begin{align*}
B_{2}(t)= & \int\int\big(g_{t}(x+y)-g_{t}(x)\big)x^{2}\nu(\d x)\PP_{\Delta}(\d
y)\\
= &
\int\big((g_{t}(-\bull)\ast(x^{2}\nu))(-y)-(g_{t}(-\bull)\ast(x^{2}
\nu))(0)\big)\PP_{\Delta}(\d y).
\end{align*}
For the sake of brevity we define $h_{t}(y):=(g_{t}(-\bull)\ast(x^{2}\nu))(y)$.
We decompose the integration domain into the neighbourhood of the
origin $(-U,U)$ and the tails $\{y:|y|\ge U\}$. For small $y$ the
uniform H\"older regularity of $h_{t}(y)$, for $|y|<U$, as well as
$\E[|X_{1}|^{2}]\lesssim\Delta$ and Jensen's inequality yield for $s\le1$
\[
\sup_{t\in\R}\Big|\int_{|y|<U}\big(h_{t}(-y)-h_{t}(0)\big)\PP_{\Delta}(\d
y)\Big|\lesssim\int_{\R}|y|^{s}\PP_{\Delta}(\d
y)=\E[|X_{1}|^{s}]\lesssim\Delta^{s/2}.
\]
and for $s>1$ with $x_y\in[-y,0]$ an intermediate point from the mean value
theorem
\begin{align*}
 &\sup_{t\in\R}\Big|\int_{|y|<U}\big(h_{t}(-y)-h_{t}(0)\big)\PP_{\Delta}(\d
y)\Big|\\
&\le
\sup_{t\in\R}\Big|\int_{|y|<U}\big(h_{t}'(x_y)-h_{t}'(0)\big)y\PP_{\Delta}(\d
y)\Big|
+
\sup_{t\in\R}\Big|\int_{|y|<U}h_{t}'(0)y\PP_{\Delta}(\d
y)\Big|\\
&\le
\sup_{t\in\R}\Big|\int_{|y|<U}|y|^s\PP_{\Delta}(\d
y)\Big|
+
\Big|\int_{\R}y\PP_{\Delta}(\d
y)\Big|\sup_{t\in\R}\left|h_{t}'(0)\right|
+
\Big|\int_{|y|\ge U}y^2 U^{-1}\PP_{\Delta}(\d
y)\Big|\sup_{t\in\R}\left|h_{t}'(0)\right|\\
&\lesssim\Delta^{s/2}+\Delta+\Delta\lesssim\Delta^{s/2}.
\end{align*}
For the tails we conclude from
$\sup_{t}\|h_{t}\|_{\infty}\le\|\rho\|_{\infty}\int x^{2}\nu(\d x)$
and Markov's inequality
\[
\sup_{t\in\R}\Big|\int_{|y|\ge U}\big(h_{t}(-y)-h_{t}(0)\big)\PP_{\Delta}(\d
y)\Big|\lesssim\PP_{\Delta}(|X_{1}|\ge U)\le
U^{-2}\E[|X_{1}|^{2}]\lesssim\Delta.
\]
The previous two estimates finally yield
$\sup_{t}|B_{2}(t)|\lesssim\Delta^{s/2}$.
\end{proof}

\subsection{Proof of Theorem \ref{thm:UCLTnaive2}}

We only prove the case $V=(-\infty, -\zeta]$, the general case follows from
symmetry arguments that are left to the reader. We use decomposition
(\ref{eq:decomp}) and apply Theorem \ref{main1} -- with $m=1$ and $\rho$
suitably chosen such that $\rho(x)=x^{-2}$ for all $x \in (-\infty, -\zeta]$ --
to the stochastic term $S(t)$. For our choice of $\Delta_n$ the bias term $B(t)$
is negligible in the asymptotic distribution in view of Proposition 2.1 in
\cite{Figueroa-Lopez2011} (which holds also for unbounded $V$ separated away
from the origin, as inspection of that proof shows).

\subsection{Proof of Theorem \ref{thmUniform}}

For Theorem~\ref{thmUniform}\ref{away} we choose a suitable $\rho$ such that
$\rho(x)=x^{-2}$ on $V$ and we restrict to the case $(-\infty,-\zeta]$ since
the proof can be easily extended to cover the general case by symmetry
arguments.
We use the decomposition (\ref{eqErrorDecomp}). The third term is negligible in
view (\ref{eq:sigmarate}). Recalling
\[
  g_t(x)=\rho(x)\1_{(-\infty,t]}(x),\quad t\in\R.
\]
the following result shows that the deterministic approximation error is negligible in the asymptotic distribution of
$\sqrt{n\Delta_n}(\hat N_n-N_\rho)$ whenever $h^{s}=o(1/\sqrt{n\Delta_n})$, valid for our choice of $\Delta_n$.
\begin{proposition}\label{propBias}
   Suppose $x^2\nu$ is a finite measure satisfying Assumption~\ref{assProc}(d).
If the kernel satisfies \eqref{propKernel} with order $p\ge s$, then
  \[
    \Big|\int_{\R} g_{t}(x)\big(K_h\ast\big(y^2\nu(\d y)\big)-x^2\nu\big)(\d
x)\Big|\lesssim c_t h^{s},
  \]
  with constants independent of $t$.
\end{proposition}
\begin{proof}
  Using Fubini's theorem,
    \begin{align}\label{eqRewrBias}
    \int_{\R} g_{t}(x)\big(K_h\ast\big(y^2\nu(\d y)\big)-x^2\nu\big)(\d x)
    =K_h\ast g_{t}(-\bull)\ast(x^2\nu)(0)-g_{t}(-\bull)\ast(x^2\nu)(0).
  \end{align}
The result now follows from Assumption \ref{assProc}(d) and a standard Taylor
expansion argument using the order $p$ of the kernel.
\end{proof}

\smallskip

The second, stochastic, term in (\ref{eqErrorDecomp}) can be reduced to the linear term from Proposition \ref{propRem}, which is proved as follows:

\begin{proof}[Proof of Proposition~\ref{propRem}]
To linearise
$\psi''-\hat\psi_n''=-\Delta^{-1}\log(\phi_{\Delta,n}/\phi_\Delta)''$, we set
$F(y)=\log(1+y)$, $\eta=(\phi_{\Delta,n}-\phi_\Delta)/\phi_\Delta$, and use
\begin{align*}
  (F\circ \eta)''(u)&=F'(\eta(u))\eta''(u)+F''(\eta(u))\eta'(u)^2\\
&=F'(0)\eta''(u)+O\Big(\norm{F''}_\infty\Big(\norm{\eta}_\infty\norm{\eta''}
_\infty
  +\norm{\eta'}_\infty^2\Big)\Big).
\end{align*}
On the event $\Omega_n:=\{\sup_{|u|\le
1/h}\abs{(\phi_{\Delta,n}-\phi_\Delta)(u)/\phi_\Delta(u)}\le 1/2\}$ we thus
obtain
\begin{align*}
  &\sup_{\abs{u}\le h^{-1}}
\big|\log(\phi_{\Delta,n}/\phi_\Delta)''(u)-(\phi_\Delta^{-1}(\phi_{\Delta,n}
-\phi_\Delta))''(u)\big|\\
&\qquad=O\Big(\norm{\eta}_{\ell^\infty[-h^{-1},h^{-1}]}\norm{\eta''}_{\ell^\infty[-h^{
-1},h^{-1}]} +\norm{\eta'}_{\ell^\infty[-h^{-1},h^{-1}]}^2\Big).
\end{align*}
To estimate $\norm{\eta^{(k)}}_{\ell^\infty[-h^{-1},h^{-1}]},k=0,1,2$, we note
$\abs{\psi'(u)}\lesssim 1+\abs{u}$, $\abs{\psi''(u)}\lesssim 1$ and $h\gtrsim
\Delta^{1/2}$
\[ \sup_{u\in[-h^{-1},h^{-1}]}\abs{(\phi_\Delta^{-1})'(u)}\lesssim \Delta
h^{-1}\lesssim\Delta^{1/2},\quad
\sup_{u\in[-h^{-1},h^{-1}]}\abs{(\phi_\Delta^{-1})''(u)}\lesssim
\Delta^2h^{-2}+\Delta\lesssim\Delta.
\]
Moreover, from Theorem 1 by \cite{kappusReiss2010} we know that under our moment
assumption on $\nu$ (for $k=0,1,2$ and any $\delta>0$)
\begin{equation}\label{eqJM}
  \norm{(\phi_{\Delta,n}-\phi_\Delta)^{(k)}}_{\ell^\infty[-h^{-1},h^{-1}]}
=O_P(n^{-1/2}\Delta^{(k\wedge 1)/2}(\log h^{-1})^{(1+\delta)/2}).
\end{equation}
This yields for $k=0,1,2$
\begin{align*}
 \|\eta^{(k)}\|_{\ell^\infty[-h^{-1},h^{-1}]}&=O_P\big(n^{-1/2}\Delta^{k/4}(\log
h^{-1})^{(1+\delta)/2}\big).
\end{align*}
In combination with $n(\log h^{-1})^{-1-\delta}\gtrsim
n\Delta^{3(1+\delta)/4}\to\infty$ for $\delta\in(0,1/3)$ and $|1/\varphi_\Delta|
\lesssim 1$ on $[-1/h_n, 1/h_n]$ the bound \eqref{eqJM} shows also
$\PP(\Omega_n)\to 1$ and then
\[ \sup_{\abs{u}\le h^{-1}}
\abs{\hat\psi_n''(u)-\psi''(u)-\Delta^{-1}(\phi_\Delta^{-1}(\phi_{\Delta,n}
-\phi_\Delta))''(u)}
=O_P(n^{-1}\Delta^{-1/2}\log(h^{-1})^{1+\delta}).
\]
We decompose the linearised stochastic error into
\[
  (\phi_\Delta^{-1}(\phi_{\Delta,n}-\phi_\Delta))''
  =\phi_\Delta^{-1}(\phi_{\Delta,n}-\phi_\Delta)''
   +2(\phi_\Delta^{-1})'(\phi_{\Delta,n}-\phi_\Delta)'
   +(\phi_\Delta^{-1})''(\phi_{\Delta,n}-\phi_\Delta).
\]
By the previous estimates we have
\begin{align*}
\sup_{\abs{u}\le
h^{-1}}\abs{(\phi_\Delta^{-1})'(\phi_{\Delta,n}-\phi_\Delta)'}(u)&=O_P(
\Delta h^{-1}n^{-1/2}\Delta^{1/2}(\log h^{-1})^{(1+\delta)/2}),\\
\sup_{\abs{u}\le
h^{-1}}\abs{(\phi_\Delta^{-1})''(\phi_{\Delta,n}-\phi_\Delta)}(u)&=O_P(
(\Delta^2 h^{-2}+\Delta)n^{-1/2}(\log h^{-1})^{(1+\delta)/2}).
\end{align*}
Inserting the asymptotics in $h$, we conclude
\begin{align*}
  &\quad\sup_{\abs{u}\le h^{-1}}
\abs{\hat\psi_n''(u)-\psi''(u)-\Delta^{-1}\phi_\Delta^{-1}(\phi_{\Delta,n}
-\phi_\Delta)''(u)}\\
  &\le\sup_{\abs{u}\le
h^{-1}}2\Delta^{-1}\abs{(\phi_\Delta^{-1})'(\phi_{\Delta,n}-\phi_\Delta)'}(u)
  +\sup_{\abs{u}\le
h^{-1}}\Delta^{-1}\abs{(\phi_\Delta^{-1})''(\phi_{\Delta,n}-\phi_\Delta)}(u)\\
  &\qquad+O_P\big(n^{-1}\Delta^{-1/2}\log(h^{-1})^{1+\delta}\big)\\
  &=O_P\left(n^{-1/2}\Delta^{-1/2}h^{1/2}\left(\Delta
h^{-3/2}+\Delta^{3/2}h^{-5/2}+\Delta^{1/2}h^{-1/2}\right)(\log
h^{-1})^{(1+\delta)/2}\right)\\
  &\qquad+O_P\big(n^{-1}\Delta^{-1/2}\log(h^{-1})^{1+\delta}\big)\\
  &=o_P\big(n^{-1/2}\Delta^{-1/2}h^{1/2}\big).
\end{align*}
By the Plancherel formula and Cauchy-Schwarz inequality we have
\begin{align*}
&\babs{\int g_t(x)\F^{-1}\Big[\mathcal F
K_h(u)\big(\hat\psi_n''(u)-\psi''(u)-\Delta^{-1}\phi_\Delta(u)^{-1}(\phi_{\Delta
,n}-\phi_\Delta)''(u)\big)\Big]
(x)\d x}\\
 &\le
\norm{\mathcal F g_t}_{L^2}\norm{\mathcal F K_h}_{L^2}\sup_{\abs{u}\le
h^{-1}}\abs{\hat\psi_n''(u)-\psi''(u)-\Delta^{-1}\phi_\Delta(u)^{-1}(\phi_{
\Delta,n}-\phi_\Delta)''(u)}\\
&=o_P(n^{-1/2}\Delta^{-1/2}).\qedhere
\end{align*}
\end{proof}

\smallskip

Finally, to the main stochastic term
\begin{align*}
M_{\Delta,n} &=\Delta_n^{-1}\int g_t(x)
\big(\F^{-1}[\phi_{\Delta_n}^{-1}\mathcal
F K_{h_n}]\ast (x^2(\PP_{\Delta_n,n}-\PP_{\Delta_n}))\big)(\d x)\\
&=\Delta_n^{-1}\int {\cal F}^{-1}[\phi_{\Delta_n}^{-1}(-u)\mathcal F
K_{h_n}(-u){\cal F}g_t(u)](x)x^2(\PP_{\Delta_n,n}-\PP_{\Delta_n})(\d x),
\end{align*}
we apply Theorem \ref{main1}. The proof of Theorem \ref{thmUniform} is thus
complete upon verification of Assumption~\ref{multass} for the present choice of
$m$. This is achieved in the following proposition.

\begin{proposition}\label{multver} Assume that $K$ satisfies (\ref{propKernel}) for
$p \ge 2$ and that $\nu$ satisfies $\int_{\R}|x|^3\nu(\d x)<\infty$.
Let $h=h_n\to0$ and $\Delta=\Delta_n\to0$ as $n\to\infty$ with $h^3=o(\Delta)$, $h^{-1} =
O(\Delta^{-1/2})$.
Then
$m_{n,\Delta}(u):=\mathcal F K_h(u)/\phi_\Delta(u)$, $u \in \R$, satisfies Assumption~\ref{multass}.
\end{proposition}
\begin{proof}
We have $m(-u)=\overline{m(u)}$ so that $\F^{-1}m$ is real-valued.
By the compact support of $\mathcal F K$ and the assumption on $h^{-1}$ the support
assumption on $m$ is satisfied. Since $\phi_{\Delta} = e^{\Delta \psi}$, we have
$m_{\Delta, n} \to 1$ pointwise as $\Delta \to 0$, $h\to0$. Moreover, by
\eqref{eqIdent} we have $|\psi''(u)|\lesssim 1$ hence for $|u| \le
C\Delta^{-1/2}$ we have $$|\phi_\Delta(u)|= |e^{\Delta \psi(u)}| \ge e^{-\Delta
cu^2} \ge c'>0$$ uniformly in $\Delta$, and thus $m \in \ell^\infty(\R)$, using
also
$\sup_h\|\mathcal F K_h\|_\infty \le \|K\|_{L^1}$. Next
\[m' = h  \frac{i\F
[xK](h\bull)}{\phi_\Delta}  + \mathcal F K_h \frac{\Delta \psi' \phi_\Delta
}{\phi^2_\Delta}\]
so that using $xK \in L^1, |\psi'(u)| \lesssim 1+|u|,
|u| \le h^{-1}=O(\Delta^{-1/2})$ and the bound for $m$ above we see
$$|m'(u)| \lesssim (h  + \sqrt{\Delta}) \lesssim h.$$
Using $|\psi''(u)|\lesssim 1$ we further obtain
\begin{align*}
 |m''(u)| &\lesssim (h^2+h\sqrt{\Delta}+\Delta)
\lesssim h^2.
\end{align*}
On the support of $m$ we have $|u|\le h^{-1}$ so that
$\|(1+|u|)^{k}m^{(k)}\|_\infty\le c$, $k\in\{0,1,2\}$, follows.
Likewise by the support of $m$ we have $\|m'\|_{L^2}\lesssim h^{1/2}\to0$
and $\Delta^{-1/2}\|m''\|_{L^2}\lesssim \Delta^{-1/2}h^{3/2}\to0$.
\end{proof}

\subsection{Convergence of finite-dimensional distributions}

We next turn to the proof of Proposition~\ref{fidi0}.

\begin{definition} \label{admis}
  A function $g$ is called \textit{admissible} if it is of bounded variation and satisfies for all $x, u \in \R$,
  \[
    |g(x)| \lesssim 1 \wedge x^{-2},\quad |\mathcal F
g(u)|\lesssim (1+|u|)^{-1} \quad\text{and}\quad u\F[xg](u)\in \ell^\infty(\R).
  \]
\end{definition}
Note that the bound on $\mathcal F g$ follows from the bounded variation of
$g$, and that $x^2 g^2(x)$ is
of bounded variation whenever $g$ is admissible.

\begin{proposition} \label{fidi1}
 Let $g$ be admissible and suppose the conditions of Proposition \ref{fidi0} are satisfied. Then
   \[
     \sqrt{n\Delta_n} \int \F^{-1}[m(-\bull)\mathcal F
g](x)\frac{x^2}{\Delta_n}(\PP_{\Delta_n,n}-\PP_{\Delta_n})(\d x)\to^{\mathcal L}
\mathcal N(0,\sigma_g^2)
   \]
   with variance $\sigma_g^2=\int_{\R} x^4g(x)^2\nu(\d x)$.
\end{proposition}

The functions $g_t=\rho\1_{(-\infty,t]}$ are uniformly bounded in bounded
variation and are admissible with constants
independent of $t\in\R$.
The convergence of the finite dimensional distributions in Proposition
\ref{fidi0} hence follows from the Cram\'er-Wold device since linear combinations
of the functions $g_{t_1},\dots,g_{t_k}$ for $t_1,\dots,t_k\in\R$
are admissible.

For the proof of Proposition~\ref{fidi1} we will use the following lemma, whose
assumptions are in particular fulfilled for $m_{n,\Delta}$ satisfying
Assumption~\ref{multass} and for classes of functions with uniform constants in
the admissibility definition.
\begin{lemma}\label{lemMultCon}
Let $\|x^3\PP_\Delta\|\lesssim\Delta$. For $\Delta\to0$ as $n\to\infty$ let $\|m_{n,\Delta}\|_\infty$
and $\|m_{n,\Delta}'\|_\infty$ be uniformly bounded and $m_{n,\Delta}\to1$
pointwise.
If $\G$ is a class of functions such that for all $u\in\R$
\[\sup_{g\in\G}|\mathcal{F}g(u)|\lesssim (1+|u|)^{-1}, \quad
\sup_{g\in\G}\|xg(x)\|_{L^2}\lesssim 1,\]
then
  \begin{align*}
\lim_{n\to\infty}\sup_{g\in\G}\int_{\R}\Big(x^2\F^{-1}[m_{n,\Delta}
(-u)\mathcal F
g(u)](x)-x^2g(x)\Big)^2\frac{\PP_\Delta}{\Delta}(\d x)=0.
  \end{align*}
\end{lemma}
\begin{proof}
We rewrite the term with $m=m_{n,\Delta}$ as
  \begin{align}
    &\Delta^{-1}\int_{\R}\F^{-1}[\mathcal F g(u)(m(-u)-1)](x)\F^{-1}[\F
g(u)(m(-u)-1)](x)x^4\PP_\Delta(\d x)\notag\\
    =&\frac{-i}{\Delta}\int_{\R}\F^{-1}\big[\mathcal F
g(u)(m(-u)-1)\big](x)\label{eqVarRem}\\
    &\qquad\quad\times\F^{-1}\big[i\F[xg](u)(m(-u)-1)-\mathcal F
g(u)m'(-u)\big](x)x^3\PP_\Delta(\d x)\notag.
  \end{align}
  Using $\|x^3\PP_\Delta\|_\infty\lesssim \Delta$, the term \eqref{eqVarRem} can be estimated by the Cauchy-Schwarz inequality and Plancherel's identity yielding the bound
  \begin{align*}
    &\int_{\R}\Big|\F^{-1}\big[\mathcal F
g(u)(m(-u)-1)\big](x)\F^{-1}\big[i\F[xg](u)(m(-u)-1)-\F
g(u)m'(-u)\big](x)\Big|\d x\\
    \le&\frac{1}{2\pi}\big\|\mathcal F
g(u)(m(-u)-1)\big\|_{L^2}\big\|\big(i\F[xg](u)(m(-u)-1)-\F
g(u)m'(-u)\big)\big\|_{L^2}.
  \end{align*}
  The first factor converges to zero by the dominated convergence theorem because $m$ is uniformly bounded and converges pointwise to one while $|\mathcal{F}g(u)|\le C(1+|u|)^{-1}$ for all $g\in\G$. For the second factor we estimate, using that $g$ and $xg$ are uniformly bounded in $L^2(\R)$ and that $\|m\|_\infty$ and $\|m'\|_\infty$ are uniformly bounded,
  \begin{align*}
    \big\|i\F[xg](u)(m(-u)-1)-\mathcal F g(u)m'(-u)\big\|_{L^2}
    &\lesssim\|\F[xg](u)\|_{L^2}+\|\mathcal F g(u)\|_{L^2}<\infty,
  \end{align*}
  which completes the proof of the lemma.
\end{proof}

\begin{proof}[Proof of Proposition~\ref{fidi1}]
We define
\begin{equation}
 S_n - \E S_n := \frac{1}{n}\sum_{k=1}^n(
Y_{n,k}-\E[Y_{n,k}])\quad\text{with}\quad
Y_{n,k}:=\Delta^{-1/2}\F^{-1}[m(-\bull)\mathcal F g](X_k)X_k^2.
\end{equation}
We will prove the proposition for general Fourier multipliers satisfying Assumption~\ref{multass}(b), the case where $\F^{-1} m$ is a finite signed measure is similar (in fact easier) and is omitted. We will verify the conditions of Lyapunov's central limit theorem, see, e.g., \cite{Bauer1996}, Theorem 28.3 and (28.8).

  \textit{Step 1:}
  We will show that
  $\lim_{n\to\infty}\Var(Y_{n,k})=\int_{\R} x^4 g(x)^2\nu(\d x)$, noting that $Y_{n,k}$ are real valued. We estimate
  \begin{align*}
  \abs{\E[Y_{n,k}]}&=\Delta^{-1/2}\babs{\int_{\R}{\cal F}^{-1}\big[m(-u)\F
  g(u)\big](x)x^2\PP_\Delta(\d x)}\\
  &\le \Delta^{-1/2}\norm{{\cal F}^{-1}[m(-u)\F
  g(u)]}_\infty\norm{x^2\PP_\Delta}_{L^1}\\
  & \lesssim \Delta^{-1/2} \int_{-C \Delta^{-1/2}}^{C \Delta^{-1/2}}
(1+|u|)^{-1}\d u ~\E[X_1^2] \\
  &\lesssim \Delta^{-1/2}\log(\Delta^{-1})\Delta\to 0
  \end{align*}
  where we have used that $\E[X_1^2]= O(\Delta)$.
  Consequently,
$\lim_{n\to\infty}\Var(Y_{n,k})=\lim_{n\to\infty}\E[Y_{n,k}^2]$,
  which we decompose in the following way:
  \begin{align*}
    \lim_{n\to\infty}\E[Y_{n,k}^2]
    &=\lim_{n\to\infty}\Delta^{-1}\int_{\R}(\F^{-1}[m(-u)\mathcal F
g(u)](x)x^2)^2\PP_\Delta(\d x)\notag\\
    &= \lim_{n\to\infty}\Delta^{-1}\int_{\R}\Big((\F^{-1}[m(-u)\mathcal F
g(u)](x)x^2)^2-(x^2g(x))^2\Big)\PP_\Delta(\d x)\\
    &\quad+\lim_{n\to\infty}\left(\Delta^{-1}\int_{\R}(x^2g(x))^2\PP_\Delta(\d
    x)-\int_{\R}(xg(x))^2 x^2\nu(\d
x)\right)\notag\\
&\quad+\int_{\R}(xg(x))^2x^2\nu(\d x).\notag
  \end{align*}
  The last term is the claimed limit. The first limit is zero by Lemma~\ref{lemMultCon}.
  For the second limit we deduce by Lemma~\ref{weakcon}
that $(x^2 \wedge x^4) \PP_\Delta/\Delta$ converges weakly to the absolutely
continuous measure $(x^2 \wedge x^4) \nu $, and thus in particular
by the Portmanteau lemma when
integrating against the function $(x^2 \vee 1) g(x)^2$, which is of bounded
variation. This implies convergence to zero of the
second term. This shows
  $\lim_{n\to\infty}\Var(Y_{n,k})=\int x^4 g(x)^2\nu(\d x)$.

  \textit{Step 2:} We verify Lyapunov's moment condition: For some $\eps\in(0,1)$ and $S_n=\sum_{k=1}^nY_{n,k}$
  \begin{align*}
  \lim_{n\to\infty}\frac{1}{\Var(S_n)^{1+\eps/2}}\sum_{k=1}^{n}\E[|Y_{n,k}|^{
  2+\eps}]=0.
  \end{align*}
  From the previous step we know $n^{-1}\Var(S_n)=\Var(Y_{n,k})\to\sigma_g^2$
  as $n\to\infty$. Moreover, by $|x|^{4+2\eps}\lesssim|1+ix|^{2+\eps}|x|^3$,
   $\|x^3\PP_\Delta\|_\infty\lesssim\Delta$ and the Hausdorff--Young inequality
\cite[e.g., 8.30 on p.~253 in][]{folland1999}
  \begin{align*}
    \E[|Y_{n,k}|^{2+\eps}]&\lesssim
\Delta^{-1-\eps/2}\int_{\R}\big|\F^{-1}\big[m(-u)\mathcal F
g(u)\big](x)x^2\big|^{2+\eps}\PP_{\Delta}(\d x)\\
    &\lesssim\Delta^{-\eps/2}\int_{\R}\big|\F^{-1}\big[m(-u)\mathcal F
g(u)\big](x)(1+ix)\big|^{2+\eps}\d x\\
    &=\Delta^{-\eps/2}\Big\|\F^{-1}\big[m(-u)\mathcal F g(u)-m'(-u)\F
g(u)+im(-u)\F[xg](u)\big]\Big\|_{L^{2+\eps}}^{2+\eps}\\
    &\lesssim\Delta^{-\eps/2}\Big\|m(-u)\mathcal F g(u)-m'(-u)\F
g(u)+im(-u)\F[xg](u)\Big\|_{L^{(2+\eps)/(1+\eps)}}^{2+\eps}.
  \end{align*}
  By Assumption~\ref{multass}, $m$ and $m'$ are uniformly bounded, $\|\mathcal F
g(u)\|_{L^{(2+\eps)/(1+\eps)}}$ is bounded by $|{\F}g|\lesssim (1+|u|)^{-1}$
and
\begin{align*}
\quad\big\|\F[xg](u)\big\|_{L^{(2+\eps)/(1+\eps)}}
\lesssim\big\|\F[xg](u)\big\|_{L^{(2+\eps)/(1+\eps)}([-1,1])}+\big\|\F[xg]
(u)\big\|_ { L^
{ (2+\eps)/(1+\eps)}([-1,1]^c)},
\end{align*}
which are finite by $xg\in L^2(\R)$ and by $u\F[xg](u)\in \ell^\infty(\R)$,
respectively.
Consequently, $\E[|Y_{n,k}|^{
  2+\eps}]\lesssim \Delta^{-\eps/2}$, implying
  \begin{align*}
  \lim_{n\to\infty}\frac{1}{\Var(S_n)^{1+\eps/2}}\sum_{k=1}^{n}\E[|Y_{n,k}|^{
  2+\eps}]\lesssim\lim_{n\to\infty}\frac{n\Delta_n^{-\eps/2}}{n^{1+\eps/2}}
  =\lim_{n\to\infty}(n\Delta_n)^{-\eps/2}=0.
  \end{align*}

\end{proof}

\subsection{Proof of Proposition \ref{Propx1PDelta}}

\begin{proof}
For (ii) and (iii) we have $\int |x| \nu (\d x)<\infty$ and will use that the function $\psi$ in the exponent of the
L\'evy--Khintchine formula \eqref{eqLevyKhintchine} may be
written as
\[\psi(u)=-\frac{\sigma^2u^2}{2}+i\gamma_0 u+\int_{\R}(e^{iux}-1)\nu(\d
x)\quad\text{ with }\gamma_0:=\gamma-\int_{\R}x\nu(\d x).\]  For (iii) note
$x\PP_\Delta=(x\PP_\Delta^+)\ast\PP_\Delta^-+(x\PP_\Delta^-)\ast\PP_\Delta^+$
with the corresponding laws for $\nu^+,\nu^-$. It thus suffices to prove
$\norm{x\PP_\Delta^+}_\infty+ \norm{x\PP_\Delta^-}_\infty\lesssim 1$ and without
loss of generality we only consider $\PP_\Delta^+$ in the proof of case~(iii).
For~(iv) we use the same decomposition but this time the
law $\PP_\Delta^+$ corresponds to the L\'evy triplet $(0,\gamma,\nu^+)$ so that
it also incorporates the drift.
\begin{enumerate}
\item If $\sigma>0$ holds, then $\abs{\psi'(u)}\lesssim 1+\abs{u}$ implies
\[ \norm{x\PP_\Delta}_\infty \le \norm{\phi_\Delta'}_{L^1}\lesssim \int \Delta(1+\abs{u})e^{-\Delta\sigma^2u^2/2}\d u\lesssim 1.
\]

\item On the assumptions the measure $x\nu$ is finite yielding the
identity $x\PP_\Delta=\Delta (x\nu)\ast\PP_\Delta$, which implies that even
    \[ \norm{x\PP_\Delta}_\infty\le \Delta\norm{x\nu}_\infty.\]

\item Without loss of generality we suppose $\|x\nu^\pm\|_\infty=\infty$. Denote
the limit inferior in condition (iii) by $\delta>0$ and define
\[
  a_\Delta:=\inf\left\{a>0:\sup_{x>a}\Delta x\nu(x)\le\frac{4}{\delta}\right\},
\]
where $a_\Delta>0$ follows from
$\lim_{a\to0}\sup_{x>a}x\nu(x)=\|x\nu^+\|_\infty=\infty$. Since
$\|x\nu\|_{\ell^\infty(\R\setminus[-\eps,\eps])}$ is bounded for any $\eps>0$ we
deduce that $a_\Delta\downarrow0$ as $\Delta\to0$.

Let us introduce $\nu_\Delta^s:=\nu{\1}_{[0,a_\Delta]}$ and
$\nu_\Delta^c:=\nu^+-\nu_\Delta^s$. By
$\|x\nu_\Delta^c\|_\infty\le\frac{4}{\Delta\delta} $ and the argument in (ii),
applied to $\nu_\Delta^c$, the corresponding law $\PP_\Delta^c$ satisfies
$\norm{x\PP_\Delta^c}_\infty\lesssim 1$. Because of

\[x\PP_\Delta=(x\PP_\Delta^c)\ast\PP_\Delta^s+(x\PP_\Delta^s)\ast\PP_\Delta^c
    =(x\PP_\Delta^c)\ast\PP_\Delta^s+(\Delta x\nu_\Delta^s)\ast\PP_\Delta
    \]
we shall bound $\norm{\Delta x\nu_\Delta^s}_{L^1}$ and
$\norm{\PP_\Delta}_\infty$.
From the assumptions we infer $\norm{\Delta x\nu_\Delta^s}_{L^1}\lesssim
a_{\Delta}$ via
\[
  \int_0^{a_\Delta} \Delta x\nu(\d x)
  =\lim_{a\downarrow a_\Delta}\int_0^a\Delta x\nu(\d x)
  \lesssim \limsup_{a\downarrow a_\Delta}a\Delta
a\nu(a)\le\frac{4a_\Delta}{\delta}.
\]
On the other hand, by construction there is some
$a_\Delta^-\in[\frac{1}{2}a_\Delta,a_\Delta]$ such that $\Delta
a_\Delta^-\nu(a_\Delta^-)\ge4/\delta$. Together with the assumptions, and
$\|\PP_\Delta\|_\infty \le \|\phi_\Delta\|_1$, we see that for
$\eps:=a_\Delta^-$ sufficiently small, that is for $\Delta$ small, and for some
$\kappa\in(2,4)$
\begin{align*}
    a_\Delta\norm{\PP_\Delta}_\infty &\le 2a_\Delta^-\int_{-\infty}^\infty
e^{-\frac{\Delta }{\kappa} u^2\int_0^{1/\abs{u}} x^2\nu(\d x)}\d u\\
    &=2\int_{-\infty}^\infty e^{-\frac{\Delta}{\kappa}
(v/a_\Delta^-)^{2}\int_0^{a_\Delta^-/\abs{v}} x^2\nu(\d x)}\d v\\
    &\le 4+2\int_{\abs{v}>1} e^{-\frac{\delta}{\kappa}\Delta
a_\Delta^-\nu(a_\Delta^-)\log(\abs{v})}\d v\\
&\le4+4\int_{1}^\infty v^{-4/\kappa}\d v\sim 1,
\end{align*}
which together with the bound on $\norm{\Delta x\nu_\Delta^s}_{L^1}$ yields the
result.

\item By Theorem 27.7 in \cite{sato1999} $\PP_\Delta$ admits a Lebesgue density, hence by Fourier inversion $\norm{x\PP_\Delta}_\infty\le \norm{\phi_\Delta'}_{L^1}$ and by the
hypothesis on $\nu^+$,
we
estimate for some $\kappa>0$ and for some small $c>0$
\begin{align*}
\norm{x\PP_\Delta^+}_\infty &\le \int_{-\infty}^\infty \Delta
\babs{\int_0^\infty(e^{iux}-1)x\nu(\d x)+\gamma}e^{-\Delta \int_0^\infty
(1-\cos(ux))\nu(\d x)}\d u\\
&\lesssim \int_{-\infty}^\infty \Delta\Big(1+\int_0^\infty (\abs{u}x^2\wedge
x)\nu(\d x)\Big) e^{-\frac{\Delta}{\kappa} u^2\int_0^{1/\abs{u}}x^2\nu(\d x)}\d
u\\
&\le \int_{-\infty}^\infty \Delta\Big(1+\int_0^\infty (\abs{u}x^2\wedge x)\nu(\d
x)\Big) e^{-c\Delta \int_0^\infty (\tfrac{u^2}2x^2\wedge\abs{u}x)\nu(\d x)}\d u.
\end{align*}
The derivative of the exponent is given by
$-c\Delta\sgn(u)\int_0^\infty (\abs{u}x^2\wedge x)\nu(\d x)$  such that the last
line of the display is bounded by
\[ \int_{-\infty}^\infty  \Delta e^{-c\Delta \int_0^\infty (\tfrac{u^2}2 x^2\wedge\abs{u}x)\nu(\d x)}\d u +2/c.
\]
From $\abs{u}\int_0^{1/u}x^2\nu(\d x)\gtrsim 1$  we infer that the integral is
at most of order $\int \Delta e^{-\Delta \abs{u}}\d u\thicksim 1$ and the result
follows.\qedhere
\end{enumerate}
\end{proof}

\section{Proof of Theorem \ref{main1}} \label{mainproof}

We recall $g_t(x)=\rho(x)\1_{(-\infty,t]}(x)$ and hence
\begin{align*}
 \mathbb{G}_n(t)=\sqrt{n}\int_{\R}\Delta^{-1/2}x^2\F^{-1}[m(-u)\mathcal F
g_t(u)](x)(\PP_{\Delta,n}-\PP_\Delta)(\d x), \quad t\in\R.
\end{align*}
By Proposition~\ref{fidi0} and Theorem~1.5.7 in \cite{vanderVaartWellner1996} it suffices to show that there
is a semimetric $d$ such that $(\R,d)$ is totally bounded and for every
$\gamma>0$ we have
\begin{align}\label{asymptotic_equicontinuity}
 \lim_{\delta\to0}\limsup_{n\to\infty}\Pr \left(\sup_{s,t \in \R: d(s,t)\le\delta}
|\mathbb{G}_n(s)-\mathbb{G}_n(t)|>\gamma\right)=0.
\end{align}
We note that $\mathbb G_n$ equals a triangular array of empirical processes $\sqrt n (\PP_{\Delta, n}-\PP_{\Delta})$ indexed by the class
\begin{align*}
 \tilde \G_n&:=\{\tilde g_t(x):t\in\R\},\\
 \tilde g_t(x)&:=\Delta^{-1/2}x^2\F^{-1}[m(-\bull)\mathcal F g_t(\bull)](x).
\end{align*}

\subsection{Equicontinuity and a change of metric}

For $t\le0$ we decompose $\tilde g_t$ into the three terms
\begin{align}
 \tilde g_t^{(1)}(x)&:=\Delta^{-1/2}x^2\F^{-1}[m(-u)\F [(\rho(\bull)
-e^{\bull-t}\rho(t)
)\1_{(-\infty,t]}(\bull)](u)](x),\label{eqtildeg1}\\
 \tilde g_t^{(2)}(x)&:=\Delta^{-1/2}x\F^{-1}[m(-u)\F[te^{\bull-t}\rho(t)
\1_{(-\infty,t]}(\bull)](u)](x),\label{eqtildeg2}\\
\tilde g_t^{(3)}(x)&:=\tilde g_t(x)-\tilde g_t^{(1)}(x)-\tilde
g_t^{(2)}(x).\label{eqtildeg3}
\end{align}
Heuristically speaking the main difficulties arise from the fact that $\1_{(-\infty, t]}$ is nonintegrable on $\R$ and discontinuous at $t$. The above decomposition separates the jump-discontinuity from the non-integrable part, and the third term collects the remainder without discontinuity or integrability issues. We refer to the second term as the `critical term' since it is not regular
enough to be treated by the usual metric entropy techniques.

For $t>0$ we replace $e^{y-t}\rho(t)\1_{(-\infty,t]}(y)$ by
$-e^{t-y}\rho(t)\1_{(t, \infty)}(y)$, and the proof below proceeds
with only notational changes. We thus restrict to $t \in (-\infty,0]$.

By the triangle inequality it suffices to show asymptotic
equicontinuity for the empirical processes
indexed by the three terms in the above decomposition
separately with appropriate metrics $d^{(i)}$, and then
\eqref{asymptotic_equicontinuity} holds with the overall
metric $d=\max _i d^{(i)}$ equal to the maximum
of the three metrics $d^{(i)}, i =1,2, 3$. In view of the variance structure
of the limiting process $\mathbb G$ it is natural to choose the semimetrics
\[d^{(i)}(s,t) = \sqrt{\int_{\R} (g_s^{(i)}-g^{(i)}_t )^2(x) \nu(\d x) },\quad
i=1,2,3,\]
where
\begin{align}
 g_t^{(1)}(x)&:=x^2(\rho(x)-e^{x-t}\rho(t)
)\1_{(-\infty,t]}(x)\label{eqg1},\\
g_t^{(2)}(x)&:=xte^{x-t}\rho(t)
\1_{(-\infty,t]}(x)\label{eqg2},\\
g_t^{(3)}(x)&:=x(x-t)e^{x-t}\rho(t)
\1_{(-\infty,t]}(x),\label{eqg3}
\end{align}
and we note $x^2g_t=g_t^{(1)}+g_t^{(2)}+g_t^{(3)}$.
On the other hand the covariance metric compatible with the distribution
$\PP_\Delta$ of the $X_k$'s driving the empirical process is given by the
$L^2(\PP_\Delta)$-distance. In the following we will show that a
$\delta$-increment for the limiting metric $d^{(i)}$ corresponds, for $n$ large
enough, to a $\delta$-increment in the $L^2(\PP_\Delta)$-metric on the functions
$\tilde g_t^{(i)}$. Verifying asymptotic equicontinuity for the whole process
then reduces to showing total boundedness of each subclass and that, for each
$i=1, 2, 3$, and every $\gamma>0$,
\begin{align}\label{eqEquContF}
 \lim_{\delta\to0}\limsup_{n\to\infty}\Pr \left(\sup_{\|\tilde
g_s^{(i)}-\tilde g_t^{(i)}\|_{2,\PP_\Delta}\le\delta}
\left|\sqrt n\int_{\R} (\tilde g_s^{(i)}-\tilde
g_t^{(i)})(\PP_{\Delta,n}-\PP_\Delta)(\d x)\right|>\gamma\right)=0,
\end{align}
where $\|f\|_{2,P}:=(\int |f|^2 \d P)^{1/2}$.
This will permit the application of powerful tools from empirical process theory
to control the last probabilities. Before we do this, we demonstrate the
reduction
to (\ref{eqEquContF}) for all three terms in the above decomposition
separately. We note that total boundedness of the
classes $\mathcal G^{(i)} = \{g_t^{(i)}: t \in \R\}$ for the $d^{(i)}$-metric
follows from entropy computations given in the following subsections.

\medskip

Starting with $\{\tilde g_t^{(1)}: t \le0\}$,
we note that the functions
\begin{align*}
 x^{-1}g_t^{(1)}(x)&=x(\rho(x)-e^{x-t}\rho(t))\1_{(-\infty,t]}(x)\\
&=(x\rho(x)-(x-t)e^{x-t}\rho(t)-e^{x-t}t\rho(t))\1_{(-\infty,t]}(x),\quad t\le0,
\end{align*}
are uniformly bounded and
uniformly Lipschitz
continuous.
In order to compare $d^{(1)}$ to the $L^2(\PP_\Delta)$-norm on $\{\tilde
g_t^{(1)}:t\le0\}$, we
claim
\begin{align}\label{metric_comparison}
 \sup_{s,t\le0}
\left|\int(\tilde g_s^{(1)}(x)-\tilde
g_t^{(1)}(x))^2{\PP_\Delta(\d x)}-\int(g_s^{(1)}(x)-g_t^{(1)}(x))^2\nu(\d
x)\right| \to 0
\end{align}
as $n\to\infty$. Any class of
functions that is uniformly bounded and uniformly Lipschitz continuous is a
uniformity class for weak convergence using either Theorem~1 in
\cite{BillingsleyTopsoe1967}, or the well-known fact that the BL-metric
metrises weak convergence.
So, the weak convergence in Lemma~\ref{weakcon}
yields
\begin{align*}
 \sup_{s,t\le0}\left|\int(x^{-1}g_s^{(1)}(x)-x^{-1}g_t^{(1)}(x))^2x^2\Delta^{
-1 } \PP_\Delta(\d
x)-\int(x^{-1}g_s^{(1)}(x)-x^{-1}g_t^{(1)}(x))^2x^2\nu(\d x)\right|\to0
\end{align*}
as $n\to\infty$. Next using $0<\rho(x)\le C (1\wedge x^{-2})$ and the bounded
variation of $\rho$, we see that
$\G:=\{x^{-2}(g_s^{(1)}(x)-g_t^{(1)}(x)):s,t\le0\}$ satisfies the assumption of
Lemma~\ref{lemMultCon} and hence
\begin{align}
 \sup_{s,t\le0}\int\left((\tilde g_s^{(1)}(x)-\tilde
g_t^{(1)}(x))-\Delta^{-1/2}(g_s^{(1)}(x)-g_t^{(1)}(x))\right)^2{\PP_\Delta(\d
x)}\to0\label{eqConvTildef}
\end{align}
as $n\to\infty$.
We conclude that \eqref{metric_comparison} and then also the reduction to
(\ref{eqEquContF}) holds for $\{\tilde g_t^{(1)}: t \in \R\}$.

A similar reduction for $\tilde g_t^{(2)}$  defined in \eqref{eqtildeg2} is
achieved as follows. As in \eqref{metric_comparison} we claim that
\begin{align}
 &\sup_{s,t\le0}\left|\int(\tilde g_s^{(2)}-\tilde g_t^{(2)})^2\d\PP_\Delta
-\int(g_s^{(2)}-g_t^{(2)})^2\d\nu\right|\notag\\
  &\le\sup_{s,t\le0}\left|\int(\tilde g_s^{(2)}-\tilde g_t^{(2)})^2\d\PP_\Delta
-\int(g_s^{(2)}-g_t^{(2)})^2\frac{\d\PP_\Delta}{\Delta}\right|\notag\\
&+\sup_{s,t\le0}\left|\int(g_s^{(2)}-g_t^{(2)})^2\frac{\d\PP_\Delta}{\Delta
}
-\int(g_s^{(2)}-g_t^{(2)})^2\d\nu\right|\label{convergence_to_Levy_measure}
\end{align}
converges to zero as $n\to\infty$. To see this we observe that by
Lemma~\ref{weakcon} the
measures $(1\wedge
x^4)\Delta^{-1}\PP_\Delta(\d x)$ converge weakly to $(1\wedge x^4)\nu(\d x)$.
The limit is absolutely continuous with respect to Lebesgue measure and thus the
functions $g_t^{(2)}(x)/(1\wedge x^2)$, $t<0$, are $(1\wedge x^4)\nu(\d
x)$-almost everywhere continuous. Moreover, the functions
\begin{align}
\frac{g_t^{(2)}(x)}{1\wedge x^2}=xte^{x-t}\rho(t)\1_{(-\infty,t]}(x)\vee
\frac{t}{x}e^{x-t}\rho(t)\1_{(-\infty,t]}(x),\quad t<0,\label{eqg2x2}
\end{align}
are all contained in a bounded set of the space
of bounded variation functions and hence $\{(g_s^{(2)}(x)/(1\wedge
x^{2})-g_t^{(2)}(x)/(1\wedge x^{2}))^2:~s,t<0\}$ forms a uniformity class for
weak convergence towards $(1 \wedge x^4)\nu(\d x) \in \ell^\infty(\R)$ (after
renormalising the
measures
involved to have mass one and by Theorem 1 in
\cite{BillingsleyTopsoe1967}).
Consequently
\begin{align}
 \sup_{s,t\le0}\left|\int(g_s^{(2)}-g_t^{(2)})^2\frac{\d \PP_\Delta}{\Delta}
-\int(g_s^{(2)}-g_t^{(2)})^2\d\nu\right|\to0\label{eqUniformityClassQ}
\end{align}
as $n\to\infty$, where we recall that $g_0^{(2)}=0$.
To deal with the first term in
\eqref{convergence_to_Levy_measure} we define
\begin{align}
 \bar
g_t^{(2)}(x)&:=\Delta^{-1/2}x^2\F^{-1}[m(-u)\F[y^{-1}e^{y-t}t\rho(t)
\1_{(-\infty,t]}(y)]
(u)](x).\label{eqbarg2}
\end{align}
Lemma~\ref{lemMultCon} can be applied to the class
$\G:=\{y^{-2}(g_s^{(2)}(y)-g_t^{(2)}(y)):s,t\le0\}$ using that $y^{-2}g_t(y)$
is uniformly bounded in the space of bounded variation functions, as observed
after \eqref{eqg2x2}.
This yields
\begin{align}
 \sup_{s,t\le0}\int\left((\bar g_s^{(2)}-\bar g_t^{(2)})
-\Delta^{-1/2}(g_s^{(2)}-g_t^{(2)})\right)^2\d
\PP_\Delta\to0\label{eqComparisonqtobarq}
\end{align}
as $n\to\infty$. Therefore, \eqref{convergence_to_Levy_measure} follows from
\eqref{eqUniformityClassQ} and \eqref{eqComparisonqtobarq} if
\begin{equation}\label{eqBarTilde}
\| \bar g_t^{(2)} -
\tilde
g_t^{(2)}\|_{L^2(\PP_\Delta)}\to0
\end{equation}
uniformly in $t\le0$. To show this,
note that
\begin{align}
 \bar g_t^{(2)}(x)-\tilde g_t^{(2)}(x)&=i\Delta^{-1/2}x
\F^{-1}[m'(-u)\F [y^{-2}g_t^{(2)}(y)](u)](x),\label{eqDifference}
\end{align}
for $t<0$ and $\bar g_0^{(2)}(x)=\tilde g_0^{(2)}$.
We will use the following proposition, which is an adaptation of the
pseudo-differential operator inequality Proposition~10 in \cite{NicklReiss2012}.
We denote the $L^q$-Sobolev space for $q\in(0,\infty)$ and $s\in\N$
by $W^s_q(\R):=\{f\in L^q(\R):\sum_{k=0}^s\|f^{(k)}\|_{L^q}<\infty\}$ and
define $\|f\|_{L^2(P)}:=(\int |f|^2 \d P)^{1/2}$.
\begin{prop}\label{propPseudoLocality}
 Let $P$ be a probability measure with Lebesgue density $P$ and such that
$\|x^{2j+k} P\|_\infty<\infty$ for some $j,k\in\N$. Let $f\in L^2(\R)$ with
$\supp(f)\cap(-\delta,\delta)=\varnothing$ for some
$\delta>0$. Then for any $p,q\in[1,2]$, $s\in\{1,2\},$ and any
compactly supported function $\mu\in W^s_q(\R)$
\begin{align*}
  \|x^{j}(\F^{-1}[\mu]*f)\|_{L^2(P)}
  \lesssim \frac{\|x^{2j+k}
P\|_\infty^{1/2}}{\delta^{k/2}}\|\mu\|_{L^{2p/(2-p)}}\|f\|_{L^{p}}
   +\delta^j\|\mu^{(s)}\|_{L^q}\left\|\frac{f(y)}{y^s}\right\|_{L^q}
\end{align*}
provided that the right-hand side is finite. The constant does not depend on $\mu$, $\delta$ or $f$.
\end{prop}
\begin{proof}
  For $f\in L^2(\R)$ and $s=1,2$ we can show, as in \cite{NicklReiss2012}, the pseudo-differential operator identity
  \begin{align*}
  (\F^{-1}[\mu]*f)(x)=\left(\left(\frac{1}{(i\bull)^s}\F^{-1}\big[\mu^{(s)}\big]
  \right)*f\right)(x), \quad x\notin\supp(f).
  \end{align*}
  Let $\delta':=\delta/2$. We use H\"older's inequality, Plancherel's identity and the Hausdorff-Young inequality to conclude
  \begin{align*}
  &\quad\int |x|^{2j}|\F^{-1}[\mu]*f|^2P(\d x)\\
  &\le\|\F^{-1}[\mu]*f\|^2_{L^{2}} \|x^{2j}\d
P\|_{\ell^\infty([-\delta',\delta']^c
) }
+\|\F^{-1}[\mu]*f\|^2_{\ell^\infty([-\delta',\delta'])}\int_{-\delta'}^{\delta'}
|x|^{2j}P(\d x)\\
  &\lesssim\|\mu\mathcal F f\|_{L^2}^2\|x^{2j+k}
P\|_{\infty}(\delta')^{-k}
    +\|(x^{-s}\F^{-1}[\mu^{(s)}](x))*f\|^2_{ L^ { \infty } ( [ -\delta',\delta']
) }(\delta')^{2j}\\
  &\lesssim(\delta')^{-k}\|x^{2j+k}
P\|_{\infty}\|\mu\|_{L^{2p/(2-p)}}^2\|\mathcal F f\|^2_{L^{p/(p-1)}} \\
&\qquad+(\delta')^{2j}\|\F^{-1}[\mu^{(s)}]\|_{L^{q/(q-1)}}^2\sup_{x\in[
-\delta',\delta']} \left(\int_{\R}\frac{|f(y)|^q}{|x-y|^{sq}}\d y\right)^{2/q}\\
  &\lesssim (\delta')^{-k}\|x^{2j+k}
P\|_{\infty}\|\mu\|_{L^{2p/(2-p)}}^2\|f\|^2_{L^{p}}
    +(\delta')^{2j}\|\mu^{(s)}\|_{L^q}^2\|f(y)/y^s\|_{L^q}^2.
  \end{align*}
  The result follows by taking the square root.
\end{proof}
We apply
Proposition~\ref{propPseudoLocality} with $P=\PP_\Delta$,
$\mu=m'(-\bull)$,
$f(y)=g_t^{(2)}(y)/y^2$, $\delta= |t|$, $p=1$, $q=2$, $k=1$, $j=1$ and $s=1$.
Using
$\|x^3\PP_\Delta\|_\infty\lesssim\Delta$ we estimate
\eqref{eqDifference} for $t<0$ by
\begin{align*}
  &\quad\Delta^{-1}\|x\F^{-1}[m'(-u)\F
[y^{-2}g_t^{(2)}(y)](u)]\|_{L^2(\PP_\Delta)}^2\\
  &\lesssim
|t|^{-1}\|m'(-\bull)\|_{L^2}^2\|y^{-2}g_t^{(2)}(y)\|_{L^1}^2+t^2\Delta^{
-1}
\|m''(-\bull)\|_{L^2}^2\|y^{-3}g_t^{(2)}(y)\|_{L^2}^2\displaybreak[0]\\
  &\lesssim
|t|^{-1}\|m'(-\bull)\|_{L^2}^2\left(\int_{-\infty}^{t}y^{-1}e^{y-t}t\rho(t)
\d
y\right)^2+\Delta^{-1}
\|m''(-\bull)\|_{L^2}^2\|e^{y-t}\1_{(-\infty,t]}(y)\|_{L^2}^2\displaybreak[0
]\\
  &\lesssim
\|m'(-\bull)\|_{L^2}^2\left(\int_{-\infty}^{t}|y|^{-1}e^{y-t}|t|^{1/2}
\rho(t)
\d
y\right)^2+\Delta^{-1}
\|m''(-\bull)\|_{L^2}^2\displaybreak[0]\\
  &\lesssim
\|m'(-\bull)\|_{L^2}^2\left(\int_{-\infty}^{-1}e^{y+1}\d
y+|t|^{1/2}\int_{-1}^{t}|y|^{-1}\d y\1_{\{t>-1\}}\right)^2+\Delta^{-1}
\|m''(-\bull)\|_{L^2}^2\\
  &\lesssim
(1+\max_{t\in[-1,0)
} |t|\log(1/|t|)^2)\|m'(-\bull)\|_{L^2}^2+\Delta^{-1}
\|m''(-\bull)\|_{L^2}^2\\
  &\lesssim
\|m'(-\bull)\|_{L^2}^2+\Delta^{-1}
\|m''(-\bull)\|_{L^2}^2,
\end{align*}
which converges to zero uniformly for $t<0$ by Assumption~\ref{multass}.
Hence, tightness of the
empirical processes indexed by $\tilde g_t^{(2)}$ can be verified by
\eqref{eqEquContF} with $i=2$.

Finally, we discuss the remaining $\tilde g_t^{(3)}=\tilde g_t-\tilde
g_t^{(1)}-\tilde g_t^{(2)}$. We have
\[
  g_t^{(3)}(x)=x^2g_t(x)-g_t^{(1)}(x)-g_t^{(2)}(x).
\]
We combine Lemma~\ref{lemMultCon} for $\G:=\{g_s-g_t:s,t\le0\}$, with \eqref{eqConvTildef}, \eqref{eqComparisonqtobarq} and
\eqref{eqBarTilde}
and obtain
\[
  \sup_{s,t\le0}\left|\left\|\tilde g_s^{(3)}-\tilde
g_t^{(3)}\right\|_{L^2(\PP_\Delta)}-\Delta^{-1/2}\left\|g_s^{(3)}-g_t^{(3)}
\right\|_{ L^2(\PP_\Delta)}\right|\to0.
\]
Exactly as in \eqref{eqUniformityClassQ} we infer
\begin{align*}
 \sup_{s,t\le0}\left|\int(g_s^{(3)}-g_t^{(3)})^2\frac{\d \PP_\Delta}{\Delta}
-\int(g_s^{(3)}-g_t^{(3)})^2\d\nu\right|\to0
\end{align*}
and thus we obtain the counterpart to \eqref{metric_comparison}
\begin{align*}
  \sup_{s,t\le0}\left|\int(\tilde g_s^{(3)}-\tilde g_t^{(3)})^2\PP_\Delta
(\d x)-\int(g_s^{(3)}-g_t^{(3)})^2\d\nu\right|\to0.
\end{align*}

\subsection{Asymptotic equicontinuity for the `non-critical terms'}

We next turn to verifying the asymptotic equicontinuity condition
(\ref{eqEquContF}) for the terms $\tilde g_t^{(i)}, i \in \{1,3\}$. We refer to
them as non-critical since uniform tightness of these processes can be deduced
directly from existing bracketing metric entropy inequalities for the empirical
process.

We recall standard empirical process notation such as $\|G\|_{\mathfrak
F}:=\sup_{f\in\mathfrak F}|G(f)|$ and $\|f\|_{2,P}:=(\int |f|^2 \d P)^{1/2}$.
We denote by
$H(\eps,\mathfrak{F},\|\cdot\|)$ the logarithm of the covering number
$N(\eps,\mathfrak{F},\|\cdot\|)$ and by
$H_{[\,]}(\eps,\mathfrak{F},\|\cdot\|)$ the logarithm of the covering number
under bracketing $N_{[\,]}(\eps,\mathfrak{F},\|\cdot\|)$ (see
\cite{vanderVaartWellner1996} for definitions). For a
class of functions $\mathfrak F$ we define
\begin{align*}
 \mathfrak{F}_\delta':=\{f-g:f,g\in\mathfrak{F},\|f-g\|_{2,\PP_\Delta}\le\delta
\}.
\end{align*}
We define the functions $f_t(x):=x^{-2}g_t^{(1)}(x)=(\rho(x)-e^{x-t}\rho(t)
)\1_{(-\infty,t]}(x)$ and
recall $\tilde
g_t^{(1)}(x)=\Delta^{-1/2}x^2\F^{-1}[m(-u)\mathcal F f_t(u)](x)$.
In order to show the equicontinuity condition~\eqref{eqEquContF} for $\tilde
g_t^{(1)}$ we define the corresponding classes
\begin{align*}
\tilde{\mathfrak{F}}:=\{\tilde
g_t^{(1)}:t\le0\}.
\end{align*}
We suppress in the notation the implicit dependence on $n$ through $\Delta$.
The weak derivative $D\rho$ is in $\ell^\infty(\R)$ by the Lipschitz continuity of
$\rho$. Since $\rho$ is also of bounded variation we have $D\rho\in
L^1(\R)\cap \ell^\infty(\R)\subset L^2(\R)$.
The class $\{f_t:t\le0\}$
is contained in a bounded set of the Sobolev space $W_2^1(\R)$
since the $L^2(\R)$-norms of $f_t$ and $Df_t$ are bounded.
By
boundedness of $m$ we conclude that $\F^{-1}[m(-u)\F
f_t(u)](x)$, $t\le0$, are contained in bounded subset of $W_{2}^1(\R)$,
which embeds continuously into $\ell^\infty(\R)$.
As an envelope of the class $(\tilde{\mathfrak{F}})_\delta'$ we can thus take
$F(x):=c\Delta^{-1/2}x^2$ for some $c>0$.
By Lemma~19.34 in \cite{vanderVaart1998} we have
\begin{align}\label{BracketingBound}
\E\|\sqrt{n}(\PP_{\Delta,n}-\PP_\Delta)\|_{(\tilde{\mathfrak{F}})_\delta'}
\lesssim J_{[\,]}(\delta,(\tilde{\mathfrak{F}})_\delta',L^2(\PP_\Delta)) +
\sqrt{n}\PP_\Delta F\{F>\sqrt{n}a(\delta)\},
\end{align}
where $a(\delta):=\delta/\sqrt{\log
N_{[\,]}(\delta,(\tilde{\mathfrak{F}})_\delta',L^2(\PP_\Delta)}$ and
\begin{align*}
J_{[\,]}(\delta,(\tilde{\mathfrak{F}})_\delta',L^2(\PP_\Delta)):=\int_0^\delta
\sqrt{\log
N_{[\,]}(\eps,(\tilde{\mathfrak{F}})_\delta',L^2(\PP_\Delta))}\d\eps.
\end{align*}
$\Delta^{1/2}x^{-2}(\tilde{\mathfrak{F}})_\delta'$ is contained in a bounded set
of the
Besov space $B^{s}_{22}(\R)$ for $s\le1$, which does not depend on $\Delta$ or
$\delta$.
Let $\gamma>0$ be such that $\int| x |^{4+2\gamma}\nu(\d x)<\infty$. We take
$s\in(1/2,1/2+\gamma)$.
The proof of Theorem~1 in \cite{NicklPoetscher2007} with $p=2$, $q=2$ and
$\beta=0$ yields
\begin{align*}
 H(\eps,\Delta^{1/2}x^{-2}(\tilde{\mathfrak{F}})_\delta',\|\cdot\langle x
\rangle^{-\gamma}\|_\infty)\lesssim \eps^{-1/s},
\end{align*}
where $\langle
x \rangle:=(1+x^2)^{1/2}$.
The entropy can be rewritten as
$H(\eps,\Delta^{1/2}x^{-2}(\tilde{\mathfrak{F}})_\delta',\|\cdot\langle x
\rangle^{-\gamma}\|_\infty)=H(\eps,\Delta^{1/2}(\tilde{\mathfrak{F}}
)_\delta',\|\cdot x^{-2}\langle x
\rangle^{-\gamma}\|_\infty)$.
A ball in the $\|\cdot x^{-2}\langle x
\rangle^{-\gamma}\|_\infty$-norm with centre $f$ and radius $\eps$ is a
bracket
\begin{align*}
 [f-\eps x^2 \langle x \rangle^{\gamma},f+\eps x^2\langle x \rangle^{\gamma}],
\end{align*}
whose $L^2(\PP_\Delta)$-size is given by $\|2\eps x^2\langle x
\rangle^{\gamma}\|_{2,\PP_\Delta}$. Consequently we have
\begin{align*}
 H_{[\,]}(\eps\|2x^2\langle x
\rangle^{\gamma}\|_{2,\PP_\Delta},\Delta^{1/2}(\tilde{\mathfrak{F}})_\delta',
L^2(\PP_\Delta))\le
H(\eps,\Delta^{1/2}(\tilde{\mathfrak{F}})_\delta',\|\cdot x^{-2}\langle x
\rangle^{-\gamma}\|_\infty)\lesssim \eps^{-1/s}.
\end{align*}
By Theorem~1.1 in \cite{Figueroa-Lopez2008} (see also \cite{FLH09}) we have
$\Delta^{-1/2}2\eps\|x^2\langle
x
\rangle^{\gamma}\|_{2,\PP_\Delta}\to2\eps\|x^2\langle x
\rangle^{\gamma}\|_{2,\nu}$ as $n\to\infty$. We obtain by a rescaling that
\begin{align*}
 H_{[\,]}(\eps,(\tilde{\mathfrak{F}})_\delta',
L^2(\PP_\Delta))
\lesssim \eps^{-1/s}.
\end{align*}
Taking $s>1/2$ we conclude that the entropy integral
$J_{[\,]}(\delta,(\tilde{\mathfrak{F}})_\delta',L^2(\PP_\Delta))$ is finite and
tends to zero as $\delta\to0$.
To show that the left hand side of \eqref{BracketingBound} tends to zero, we
first ensure that the entropy integral is small by choosing $\delta>0$. Upon
fixing $\delta$ and thus for fixed $a(\delta)$ bounded away from zero uniformly
in $\Delta$, we choose $n$ large enough
such that the second term is small.
We recall that we have taken the
envelopes to be $F(x)=c\Delta^{-1/2}x^2$.
We bound
\begin{align*}
 \sqrt{n}\PP_\Delta F\{F>\sqrt{n}a(\delta)\}
&\lesssim
\sqrt{n}\Delta^{-1/2}\int
x^2\1_{\{x^2>\sqrt{n\Delta}\,a(\delta)/c\}}\PP_\Delta(\d x)\\
&\lesssim \Delta^{-1}\int
x^4\1_{\{x^2>\sqrt{n\Delta}\,a(\delta)/c\}}\PP_\Delta(\d x),
\end{align*}
where we multiplied by $cx^2/(\sqrt{n\Delta}\,a(\delta))>1$.
For $M$ large enough
$\int x^4 \1_{\{x^2>M\}}\nu (\d x)$ is small.
Since $\Delta^{-1}\int x^4 \1_{\{x^2>M\}}\PP_\Delta(\d x)\to\int x^4
\1_{\{x^2>M\}}\nu(\d x)$ by Theorem~1.1 in \cite{Figueroa-Lopez2008},
$n\Delta\to\infty$ as $n\to\infty$ and $a(\delta)$ is bounded away from zero, we
have that
$\Delta^{-1}\int x^4\1_{\{x^2>\sqrt{n\Delta}\,a(\delta)/c\}}\PP_\Delta(\d x)$
is small for $n$ large enough.
So indeed the left hand of \eqref{BracketingBound} tends to zero as
$\delta\to0$ and $n\to\infty$ and we have shown tightness of the empirical
process indexed by $\{\tilde g^{(1)}_t:t\le0\}$.

Let us now consider the terms associated to
\begin{align*}
 \tilde g_t^{(3)}(x)&=\tilde g_t(x)-\tilde g_t^{(1)}(x)-\tilde
g_t^{(2)}(x)\\
&=\Delta^{-1/2}x^2\F^{-1}[m(-u)\F [e^{y-t}\rho(t)
\1_{(-\infty,t]}(y)](u)](x)\\
&\quad-\Delta^{-1/2}x\F^{-1}[m(-u)\F[te^{y-t}
\rho(t)\1_{(-\infty,t]}(y)](u)](x)\\
&=i\Delta^{-1/2}x\F^{-1}[m'(-u)\F [e^{y-t}\rho(t)\1_{(-\infty,t]}(y)](u)](x)\\
&\quad+\Delta^{-1/2}x\F^{-1}[m(-u)\F[(y-t)e^{y-t}
\rho(t)\1_{(-\infty,t]}(y)](u)](x).
\end{align*}
The functions $(y-t)e^{y-t}
\rho(t)\1_{(-\infty,t]}(y)$ are uniformly for all $t\le0$ bounded in $L^2(\R)$
and likewise are their weak derivatives. We conclude that they are contained in
a bounded set of $B^1_{22}(\R)$. The functions
$e^{y-t}\rho(t)\1_{(-\infty,t]}(y)$,
$t\le0$, are contained in a bounded set of
$L^2(\R)$.
Assumption \ref{multass} implies, together with the Mikhlin Fourier multiplier
theorem (e.g., Corollary 4.11 in \cite{girardiWeis2003}), that $m$ is a Fourier
multiplier on every Besov space $B^s_{pq}(\R)$, $s\in\R$, $p,q\in[1,\infty]$,
and, moreover, that
$m'$ is a Fourier multiplier mapping $B^s_{pq}(\R)$ into
$B^{s+1}_{pq}(\R)$.
We see that $\Delta^{1/2}x^{-1}\tilde g^{(3)}_t(x)$, $t\le0$, are contained in a
bounded set of $B^1_{22}(\R)$. We define the class $\tilde{\mathcal
G}:=\{\tilde g^{(3)}_t:t\le0\}$.

As an envelope of the class $(\mathcal {\tilde G})_{\delta}'$ we
can take $G(x):=c\Delta^{-1/2}x$ for some constant $c>0$.
Lemma~19.34 in \cite{vanderVaart1998} yields
\begin{align}\label{BracketingBoundH}
\E\|\sqrt{n}(\PP_{\Delta,n}-\PP_\Delta)\|_{(\tilde{\mathcal{G}})_\delta'}
\lesssim J_{[\,]}(\delta,(\tilde{\mathcal{G}})_\delta',L^2(\PP_\Delta)) +
\sqrt{n}\PP_\Delta G\{G>\sqrt{n}a(\delta)\}.
\end{align}
Again by the proof of
Theorem~1 in \cite{NicklPoetscher2007} with $s=1$, $p=2$, $q=2$,
$\beta=0$ and $\gamma=1$ we have
\begin{align*}
 H(\eps,\Delta^{1/2}x^{-1}(\tilde{\mathcal G})_{\delta}',\|\cdot
\langle x\rangle^{-1}\|_\infty)\lesssim \eps^{-1}.
\end{align*}
The entropy can be rewritten as $H(\eps,\Delta^{1/2}(\tilde{\mathcal
G})_{\delta}',\|\cdot
x^{-1}\langle x\rangle^{-1}\|_\infty)$.
A corresponding $\eps$ ball is in the $L^2(\PP_\Delta)$-norm of size
$2\eps\|x\langle x \rangle\|_{2,\PP_\Delta}$.
By Theorem~1.1 in
\cite{Figueroa-Lopez2008} we have
$\Delta^{-1/2}\|x\langle x\rangle\|_{2,\PP_\Delta}\to\|x\langle x
\rangle\|_{2,\nu}+\sigma^2$ as $n\to\infty$.
Arguing as for $\tilde{\mathfrak{F}}$ we obtain
\begin{align*}
 H_{[\,]}(\eps,(\tilde{\mathcal G})_\delta',L^2(\PP_\Delta))\lesssim
\eps^{-1}.
\end{align*}
The entropy integral in \eqref{BracketingBoundH} is finite and converges to
zero
as $\delta\to0$. The second term $\sqrt{n}\PP_\Delta G\{G>\sqrt{n}a(\delta)\}$
can be treated exactly as the second term in \eqref{BracketingBound} with
$x^2$ replaced by $x$.
So the
$\lim_{\delta\to0}\limsup_{n\to\infty}$ of \eqref{BracketingBoundH} is zero and
thus
\eqref{eqEquContF} follows for the functions $\tilde g^{(3)}_t$.

\subsection{Asymptotic equicontinuity of the `critical term'}

It remains to show asymptotic equicontinuity of the empirical process indexed by
the class
\begin{align*}
 Q_n&:=\{\tilde g_t^{(2)}:t\le0\},
\end{align*}
where we recall from \eqref{eqtildeg2} that
\begin{equation}
  \tilde
g_t^{(2)}(x)=\Delta^{-1/2}x(\F^{-1}[m(-u)]\ast q_t)(x),~~
  q_t(y):=t\rho(t)e^{y-t}\1_{(-\infty,t]}(y).\label{eqQ}
\end{equation}
We refer to this term as `critical': the functions $q_t$ contain a
step-discontinuity at $t$ and controlling its interaction with the operator
$\F^{-1}[m(-\cdot)]$ needs some more elaborate techniques than in the previous
section.

We will rely on the following auxiliary result, which is a modification of
Theorem 3 in \cite{GineNickl2008}, which in itself goes back to fundamental
ideas in \cite{GineZinn1984}. It is designed to allow for
maximally growing envelopes of the empirical process, which is crucial in our
setting to allow for minimal conditions on $\Delta$. Note that indeed
Condition~\eqref{envelopes} only requires $M_n/n^{1/2}\to0$ instead of the more
stringent condition $M_n/n^{1/4} \to 0$ which was required in Theorem~3 in
\cite{GineNickl2008}.

\begin{proposition}
For every $n \in \N,$ let $X_{n,j}, j=1, \dots, n,$ be i.i.d.~from law $P_n$ on a
measurable space $(S, \mathcal B)$ and let $\eps_j$, $j=1,\dots,n,$ be i.i.d.
Rademacher random variables independent of the $X_{n,j}$'s, all defined on a
common probability space $(\Omega, \mathcal A, \Pr)$. For any sequence
$(\mathcal Q_n)_{n\ge 1}$ of classes of measurable functions $q: S \to \R$ and
\begin{equation*}
 (\mathcal Q_n)'_{r}:=\{q-q':q,q'\in\mathcal Q_n,
\|q-q'\|_{2,P_n}\le r_n\}, n \in \N,
\end{equation*}
suppose the following conditions are satisfied for some sequence $r_n\to0$ as
$n\to\infty$
\begin{enumerate}[(a)]
 \item\label{envelopes} $\sup_{q\in\mathcal Q_n}\|q\|_\infty\le M_n$ for a sequence $M_n$ such that $n r_n^{2}{M_n}^{-2}\to\infty$.
 \item\label{r_increments} \begin{equation*}
\left\|\frac{1}{\sqrt{n}}\sum_{j=1}^n\eps_j q(X_{n,j})\right\|_{(\mathcal
Q_n)'_{r}} = o_P(1)
       \end{equation*}
as $n\to\infty$.
\item\label{entropy_bound} There exists $n_0\in\N$ such that for all $n\ge n_0$
      \begin{equation*}
       23  H(r_n,\mathcal Q_n,L^2(P_n))\le n r_n^{2}{M_n}^{-2}.
      \end{equation*}
\item\label{uniform_entropy_integrals} \begin{equation*}
       \lim_{\delta\to0}\limsup_{n\to\infty}\int_0^\delta
\sqrt{H( \eps,\mathcal Q_n, L^2(P_n))}\d\eps=0
      \end{equation*}
\end{enumerate}
Then for all $\gamma>0$
\begin{align*}
  \lim_{\delta\to0}\limsup_{n\to\infty}\Pr \left(\sup_{q,q' \in \mathcal Q_n: \|q-q'\|_{2,P_n}\le\delta}
  \left|\sqrt n\int_S (q- q')\left(\frac{1}{n}\sum_{j=1}^n \delta_{X_{n,j}}-P_n\right)(\d x)\right|>\gamma\right)=0.
\end{align*}
\end{proposition}
\begin{proof}
Let $\gamma >0$ be given. We sometimes omit to mention $q,q' \in \mathcal Q$ to
expedite notation. By Lemma~11.2.6 in \cite{dudley1999} we have for
$\delta\in(0,\gamma/\sqrt{2})$
\begin{align*}
&{\Pr}\left(\sup_{\|q-q'\|_{2,P_n}\le\delta}
\frac{1}{\sqrt{n}}\left|\sum_{j=1}^{n}
q(X_{n,j})-q'(X_{n,j})-\E[q(X_{n,j})]+\E[q'(X_{n,j})]\right|>\gamma\right)\\
&\le4{\Pr}\left(\sup_{ \|q-q'\|_{2,P_n}\le\delta}
\frac{1}{\sqrt{n}}\left|\sum_{j=1}^{n}\eps_j(q(X_{n,j})-q'(X_{n,j}
))\right|>\frac{\gamma-\sqrt{2}\delta}{2}\right),
\end{align*}
where $\eps_j$ are Rademacher random variables independent of the $(X_{n,j})$,
all defined on a large product probability space.
Since $\gamma$ is given and
$\delta$ tends to zero, we can choose $\delta$ small enough such that
$\delta<\gamma/2$. Hence it suffices to show for all $\gamma>0$ that
\begin{align*}
 \lim_{\delta\to0}\limsup_{n\to\infty}{\Pr}\left(\sup_{ q\in(\mathcal
Q_n)'_\delta}
\frac{1}{\sqrt{n}}\left|\sum_{j=1}^{n}\eps_j
 q(X_{n,j})\right|>3\gamma\right)=0.
\end{align*}
Let $\mathcal H=\mathcal H_n$ be a maximal collection of functions
$h_1,\dots,h_m$ in $\mathcal Q_n$ such that $\|h_j-h_k\|_{2,\PP_\Delta}>r_n$ if
$j\neq k$. The closed balls with centres $h_1,\dots,h_m$ of radius $r_n$ cover
$\mathcal Q_n$.
We define \[\mathcal H'_\delta:=\{g-h:g,h\in \mathcal H,
\|g-h\|_{2,P_n}\le \delta\}.\] For $n$ large enough such that
$r_n<\delta/2$ we have
\begin{align}
 &{{\Pr}}\left(\sup_{q\in(\mathcal Q_n)'_\delta}
\frac{1}{\sqrt{n}}\left|\sum_{j=1}^{n}\eps_j
q(X_{n,j})\right|>3\gamma\right)\notag\\
&\le 2{\Pr}\left(\sup_{q\in(\mathcal Q_n)'_r}
\frac{1}{\sqrt{n}}\left|\sum_{j=1}^{n}\eps_j
 q(X_{n,j})\right|>\gamma\right)+{\Pr}\left(\max_{h\in \mathcal H_{2\delta}'}
\frac{1}{\sqrt{n}}\left|\sum_{j=1}^{n}\eps_j
h(X_{n,j})\right|>\gamma\right).\label{eqGridH}
\end{align}
By condition~\eqref{r_increments} the first term tends to zero.
To control the
second term we define the event
\begin{equation*}
 A_n:=\left\{\max_{h\in\mathcal
H_{2\delta}'\backslash\{0\}}\frac{\sum_{j=1}^n
h^2(X_{n,j})}{nP_n h^2}<2\right\},
\end{equation*}
where we used the notation $P_n f:=\int_S f \d P_n$ for functions $f:S\to\R$.
Using Markov's inequality the second term in \eqref{eqGridH} can be bounded by
\begin{align}
&{\Pr}\left(\max_{h\in \mathcal H_{2\delta}'}
\frac{1}{\sqrt{n}}\left|\sum_{j=1}^{n}\eps_j
h(X_{n,j})\right|>\gamma\right)\notag\\
&\le
{\Pr}(A_n^c)+\frac{1}{\gamma}\E_X\E_\eps\left[\left\|\frac{\sum_{j=1}^{n}\eps_j
h(X_{n,j})}{\sqrt{n}} \right\|_{\mathcal H_{2\delta}'}\1_ { A_n } \right
].\label{conditioned_on_A}
\end{align}
The number of elements in $\mathcal H'_{2\delta}$ is bounded by
\begin{equation}\label{size_H}
 \#\mathcal H'_{2\delta}\le \exp(2 H(r_n,\mathcal Q_n,L^2(P_n))).
\end{equation}
For a single $h\in\mathcal H_{2\delta}'\backslash\{0\}$ we have, using
Bernstein's inequality,
\begin{align*}
{\Pr}\left(\frac{\sum_{j=1}^n
h^2(X_{n,j})}{nP_n h^2}\ge2\right)
&={\Pr}\left(\sum_{j=1}^{n}(h^2(X_{n,j})-P_n h^2)\ge n P_n
h^2\right)\\
&\le \exp\left(-\frac{n^2 (P_n h^2)^2}{2n P_n h^4+8M_n^2
n P_n h^2/3}\right)\\
&\le\exp\left(-\frac{nP_n h^2}{11 M_n^2}\right).
\end{align*}
Combining the last bound and \eqref{size_H} we obtain
\begin{align*}
 \Pr(A_n^c)\le\exp\left(2 H(r_n,\mathcal Q_n,L^2(P_n))-\frac{n r_n^2}{11
M_n^2}\right)\to0
\end{align*}
by condition~\eqref{envelopes} and~\eqref{entropy_bound}.
It remains to show that the second term in~\eqref{conditioned_on_A} converges
to zero. Conditional on the $X_{n,j}$'s the process
\begin{equation*}
 Z(h):=\frac{1}{\sqrt{n}}\sum_{j=1}^{n}\eps_j h(X_{n,j}), \quad h\in\mathcal H,
\end{equation*}
is subgaussian. Let $h,h'\in\mathcal H$ such that $h-h'\in\mathcal
H_{2\delta}'\backslash\{0\}$. On the event $A_n$ we have
\begin{align*}
 d_Z(h,h')^2&:=\E_\eps\left[\left(\frac{1}{\sqrt{n}}\sum_{j=1}
^n\eps_j h(X_{n,j})-\frac{1}{\sqrt{n}}\sum_{j=1}
^n\eps_j
h'(X_{n,j})\right)^2\right]\\
&=\frac{1}{n}\sum_{j=1}^n(h(X_{n,j})-h'(X_{n,j}
))^2<2 P_n(h-h')^2.
\end{align*}
Especially we have on $A_n$ that $\|h-h'\|_{2,P_n}<\eps$ implies
$d_Z(h,h')<\sqrt{2}\eps$ for all $\eps>0$ with $\eps\le2\delta$ and for all
$h,h'\in\mathcal H$.
We define
$\psi_2(x):=\exp(x^2)-1$ and the norm
$\|\xi\|_{\psi_2}:=\inf\{c>0:\E[\psi_2(|\xi|/c)]\le1\}$. By~(4.3.3) in
\cite{delaPenaGine1999} there is a constant $c>0$ such that
$\E[|\xi|]\le c \|\xi\|_{\psi_2}$. So we obtain the bound
\begin{align*}
 \E_\eps\left[\left\|\frac{\sum_{j=1}^{n}\eps_j
h(X_{n,j})}{\sqrt{n}} \right\|_{\mathcal H_{2\delta}'}\1_ { A_n } \right
]\le c \left\|\sup_{d_Z(h,h')<2\sqrt{2}\delta}\left|\frac{\sum_{j=1}^{n}\eps_j
(h(X_{n,j})-h'(X_{n,j}))}{\sqrt{n}} \right|  \1_{A_n} \right\|_{\psi_2}.
\end{align*}
Next we apply Dudley's theorem in the form of Corollary~5.1.6 and Remark~5.1.7
in \cite{delaPenaGine1999} to the process $Z$. This yields a constant $K$ such
that
\begin{align*}
 \E_\eps\left[\left\|\frac{\sum_{j=1}^{n}\eps_j
h(X_{n,j})}{\sqrt{n}} \right\|_{\mathcal H_{2\delta}'}\right
]\1_ { A_n }
&\le
K \int_0^{2\sqrt{2}\delta}\left(\log(N(\eps,\mathcal
H_n,d_Z))\right)^{1/2}\d\eps \1_{A_n}\\
&\le
K \int_0^{2\sqrt{2}\delta}\left(\log(N(\eps/\sqrt{2},\mathcal
H_n,L^2(P_n)))\right)^{1/2}\d\eps \1_{A_n},\\
&\le\sqrt{2}
K \int_0^{2\delta}\left(\log(N(\eps,\mathcal Q_n,L^2(P_n)))\right)^{1/2}\d\eps,
\end{align*}
a bound independent of $X$. In order to complete the proof we take expectation with respect to $X$,
consider
the limit $\lim_{\delta\to0}\limsup_{n\to\infty}$ of the expression and
apply condition~\eqref{uniform_entropy_integrals}.
\end{proof}

\medskip

To proceed with the tightness proof for the critical term we will show conditions~\eqref{envelopes}
to~\eqref{uniform_entropy_integrals} for
$r_n:=\log(1/\Delta_n)^{-\alpha}$, $\alpha\in(1/2,1)$, and for the class $
Q_n=\{\tilde
g_t^{(2)}:t\le0\}$ defined above.

\eqref{envelopes}
We rewrite
\begin{align}
 \tilde
g^{(2)}_t&=\Delta^{-1/2}i\F^{-1}[m'(-u)\F[t\rho(t)e^{y-t}\1_{(-\infty,t]}(y)](u)
] \notag\\
&\quad+\Delta^{-1/2}\F^{-1}[m(-u)\F[t\rho(t)ye^{y-t}\1_{(-\infty,t]}(y)](u)]
\notag\displaybreak[0]\\
&=\Delta^{-1/2}i\F^{-1}[m'(-u)\F[t\rho(t)e^{y-t}\1_{(-\infty,t]}(y)](u)]\notag\\
&\quad+\Delta^{-1/2}\F^{-1}[m(-u)\F[t\rho(t)(y-t)e^{y-t}\1_{(-\infty,t]}(y)](u)]
\label{eqDecomposition}\\
&\quad+\Delta^{-1/2}\F^{-1}[m(-u)\F[t^2\rho(t)e^{y-t}\1_{(-\infty,t]}(y)](u)],
\notag
\end{align}
where the last step also shows that the bounded variation norm of
\(
t\rho(t)ye^{y-t}\1_{(-\infty,t]}(y)
\)
is bounded uniformly in $t\le0$.
If $\F^{-1}[m]$, $\F^{-1}[m']$ are finite signed measures as in
Assumption~\ref{multass}(a), then the bounded
variation norms of
$\F^{-1}[m'(-\bull)]*(t\rho(t)e^{y-t}\1_{(-\infty,t]}(y))$ and
$\F^{-1}[m(-\bull)]*(t\rho(t)ye^{y-t}\1_{(-\infty,t]}(y))$ are bounded uniformly
in
$t\le0$
and
\begin{align*}
 \|\tilde g^{(2)}_t\|_\infty\le\|\tilde g^{(2)}_t\|_{BV}\lesssim \Delta^{-1/2},
\end{align*}
where $\|f\|_{BV}$ denotes the bounded variation
norm equal to the sum of the $\ell^\infty$-norm of $f$ and the usual total
variation norm of the weak derivative $Df$.
For $m$ supported in $[-C\Delta^{-1/2},C\Delta^{-1/2}]$ as in
Assumption~\ref{multass}(b),
we have
\begin{equation*}
 \|\tilde g^{(2)}_t\|_\infty\lesssim \|\tilde
g^{(2)}_t\|_{B^0_{\infty,1}}\lesssim
\|\tilde g^{(2)}_t\|_{B^1_{1,1}}
\end{equation*}
and
the Fourier transform of $\tilde g^{(2)}_t$ is supported on
$[-C\Delta^{-1/2},C\Delta^{-1/2}]$. In view of the
Littlewood-Paley definition of Besov spaces we can
estimate the $B^1_{11}(\R)$-norm of $\tilde g^{(2)}_t$ by
$\log(C/\Delta^{1/2})$-times its $B^1_{1\infty}(\R)$-norm. With the Fourier
multiplier
property of $m$ and $m'$ this yields
\begin{align*}
 \|\tilde g^{(2)}_t\|_\infty &\lesssim
\log(C/\Delta^{1/2})\|\tilde g^{(2)}_t\|_{B^1_{1\infty}}\\
&\lesssim
\Delta^{-1/2}\log(1/\Delta)(\|t\rho(t)e^{y-t}\1_{(-\infty,t]}(y)\|_{B^1_{1\infty
}}+\|t\rho(t)ye^{y-t}\1_{(-\infty,t]}(y)\|_{B^1_{1\infty}})\\
&\lesssim\Delta^{-1/2}\log(1/\Delta),
\end{align*}
since the $B^1_{1\infty}(\R)$-norm of
$t\rho(t)e^{y-t}\1_{(-\infty,t]}(y)$ and $t\rho(t)ye^{y-t}\1_{(-\infty,t]}(y)$
are uniformly in $t$ bounded by
integrability and
bounded variation.
So $M_n$ can be chosen proportional to $\Delta^{-1/2}\log(1/\Delta)$ and
$nr_n^2M_n^{-2}\to\infty$ by $\log^4(1/\Delta)=o(n\Delta)$.

\eqref{r_increments} We will show condition~\eqref{r_increments} by applying a moment inequality for empirical processes under uniform
entropy bounds for $Q_n$.
We decompose $\tilde g^{(2)}_t$ according to \eqref{eqDecomposition}. Using
that $B_{11}^{1}(\R)$ embeds continuously into the space $\text{BV}$ of bounded
variation functions,
the bounds in \eqref{envelopes} show that
\begin{align}
 \|\Delta^{-1/2}\F^{-1}[m'(-u)\F[e^{y-t}\1_{(-\infty,t]}(y)](u)]\|_{BV}
&\lesssim
\Delta^{-1/2}\log(1/\Delta),\notag\\
 \|\Delta^{-1/2}\F^{-1}[m(-u)\F[(y-t)e^{y-t}\1_{(-\infty,t]}(y)](u)]\|_{BV}
&\lesssim
\Delta^{-1/2}\log(1/\Delta),\label{BV_bound}\\
 \|\Delta^{-1/2}\F^{-1}[m(-u)\F[e^{y-t}\1_{(-\infty,t]}(y)](u)]\|_{BV}
&\lesssim
\Delta^{-1/2}\log(1/\Delta),\notag
\end{align}
where we omitted the factors $t\rho(t)$ and  $t^2\rho(t)$ to obtain
translation invariant
classes. Since the functions in the class
\begin{align}
 \mathfrak{F}_n:=\{\Delta^{-1/2}\F^{-1}[m'(-u)\F[e^{y-t}\1_{(-\infty,t]}(y)](u)]
:t\le0\}
\end{align}
are of bounded variation, we can write
them as the composition of a 1-Lipschitz function after a nondecreasing
function.
The class of all translates of a nondecreasing function has VC index~2 and thus
polynomial $L^2(\QQ)$-covering numbers uniformly in all probability measures
$\QQ$ by Theorem~5.1.15 in \cite{delaPenaGine1999}. The $\eps$-covering numbers
are preserved under 1-Lipschitz transformations and thus the covering
numbers of $\mathfrak F_n$ are polynomial in
$M_n/\eps$. The $\eps$-covering numbers of $\{t\rho(t):t\in\R\}$
are
polynomial in $1/\eps$. To obtain an $\eps$-covering of the functions in the
first term of \eqref{eqDecomposition} we cover
the class
$\mathfrak{F}_n$ by balls of size $\eps/2$ and the class
$\{t\rho(t):t\in\R\}$ by balls of size $\eps/(2M_n)$. We see that the
covering numbers can be bounded by a product of two
polynomial
covering numbers and thus are polynomial in $M_n/\eps$.
Arguing in the same way for the two other terms in \eqref{eqDecomposition}
yields polynomial covering numbers for them, too.
Using that the covering numbers of $Q_n$ can be bounded by the product of the
covering numbers for the respective terms we see that the covering numbers of
$Q_n$ are polynomial in $M_n/\eps$.
By Proposition~3 in
\cite{GineNickl2009} there exists a universal constant $L>0$ such that
\begin{align*}
\E \left\|\frac{1}{\sqrt{n}}\sum_{j=1}^n\eps_j
q(X_{n,j})\right\|_{(Q_n)'_{r}}\le L
\max\left(r_n\sqrt{\log\left(\frac{M_n}{r_n}\right)},\frac{M_n}{\sqrt{n}}
\log\left(\frac{M_n}{r_n}\right)\right).
\end{align*}
Condition \eqref{r_increments} is satisfied if this maximum tends to zero.
We have
\begin{align*}
 r_n\sqrt{\log\left(\frac{M_n}{r_n}\right)}
 \lesssim\frac{\sqrt{\log\left(\log(1/\Delta)^{1+\alpha}/\Delta^{1/2}\right)}}{
\log(1/\Delta)^\alpha}
 \lesssim\frac{\sqrt{\log(\log(1/\Delta))}}{\log(1/\Delta)^\alpha}+\frac{
\sqrt { \log(1/\Delta)}}{\log(1/\Delta)^\alpha}\to0,
\end{align*}
and
\begin{align*}
 \frac{M_n}{\sqrt{n}}
\log\left(\frac{M_n}{r_n}\right)
&\lesssim\frac{\log(1/\Delta)}{\sqrt{\Delta n}}
\log\left(\frac{\log(1/\Delta)^{1+\alpha}}{\Delta^{1/2}}\right)\\
&\lesssim\frac{\log(1/\Delta)\log(\log(1/\Delta))}{\sqrt{\Delta n}}
+\frac{\log(1/\Delta)^2}{\sqrt{\Delta n}},
\end{align*}
which tends to zero by $\log^4(1/\Delta)=o(n\Delta)$.

\eqref{entropy_bound}
In order to verify \eqref{entropy_bound}, we will show that
$H(\eps,Q_n,L^2(\PP_\Delta)\lesssim
\log(\eps^{-1})$ uniformly in $n$.
Applying Proposition~\ref{propPseudoLocality}
with $j=k=1$, $p=q=2$ and $s=2$ yields
that for $\mu=m(-\bull)$
and for
all $f\in L^2(\R)$ with $\supp (f)\cap (-\delta,\delta)=\varnothing$ for some
$\delta>0$
\begin{align}
  \Delta^{-1/2}\|x(\F^{-1}[m(-u)]*f)\|_{2,\PP_\Delta}
  &\lesssim\left(
\delta^{-1/2}\|m\|_\infty+\delta^{-1}\Delta^{-1/2}\|m''\|_{L^2}\right)\|f\|_
{L^2}\notag\\
  &\lesssim \left(\delta^{-1/2}\vee
\delta^{-1}\right)\|f\|_{L^2}\label{eqApplProp}
\end{align}
where we used $\|m\|_\infty\le C$ and $\Delta^{-1/2}\|m''\|_{L^2}\to0$ by
Assumption~\ref{multass}.

Let $M\ge1$ and $\eta\in[0,1]$.
We will distinguish the three cases $s,t\le-M$, $s,t\in[-M,-\eta]$ and
$s,t\in[-\eta,0]$.

\noindent
\textbf{Case 1:}
Let $s,t\le-M$. We apply \eqref{eqApplProp} to $\delta=M$ and
$f(y):=q_s(y)-q_t(y)$ with $q_t$ defined in \eqref{eqQ}. Noting that we can
bound $\|q_t\|_{L^2}\lesssim
M^{-1}$ uniformly in $t\le0$, we obtain for $s,t\le-M$
\begin{align*}
 \|\tilde g_s^{(2)} -\tilde g_t^{(2)}\|_{2,\PP_\Delta}\lesssim M^{-3/2}.
\end{align*}

\noindent
\textbf{Case 2:}
For the second case let $-M\le s,t\le-\eta$.
We apply \eqref{eqApplProp} with $\delta=\eta$ to $f(y):=q_s(y)-q_t(y)$.
Without loss of generality we assume $s\le t$. We estimate
\begin{align*}
& \int_{-\infty}^{t}(q_s(y)-q_t(y))^2\d y\notag\\
&=\int_{-\infty}^{0}\left(s\rho(s)e^{y}-t\rho(t)e^{y+s-t}\right)^2\d
y+\int_{s}^{t}t^2\rho(t)^2e^{2(y-t)}\d
y\notag\\
&\le 2\left(s\rho(s)-t\rho(t)\right)^2\int_{-\infty}^0e^{2y}\d
y+2t^2\rho(t)^2\int_{-\infty}^0(1-e^{s-t})^2e^{2y}\d y
+t^2\rho(t)^2|s-t|\notag\\
&\lesssim |s-t|^2+|s-t|
\end{align*}
by the Lipschitz continuity of $x\rho$ and
obtain for $s,t\in[-M,-\eta]$ with $|s-t|\le1$
\begin{align*}
 \|\tilde g_s^{(2)} -\tilde g_t^{(2)}\|_{2,\PP_\Delta}\lesssim
|s-t|^{1/2}/\eta.
\end{align*}

\noindent
\textbf{Case 3:}
Let $-\eta\le s,t\le0$. We have $\|\tilde g_s^{(2)} -\tilde
g_t^{(2)}\|_{2,\PP_\Delta}\le2\sup_{t\in[-\eta,0]}\|\tilde
g_t^{(2)}\|_{2,\PP_\Delta}$. We apply Proposition~\ref{propPseudoLocality} with
$f=q_t$, $\mu=m(-\bull)$,
$\delta=|t|$, $k=1$, $j=1$, $p=2$, $q=2$ and $s=2$. We have
\begin{align*}
 |t|^{-1}\|q_t\|_{L^2}^2&=|t|\rho(t)^2\int_{-\infty}^{0}e^{2y}\d y\lesssim
|t|\quad\text{and}\\
 t^2\left\|q_t(y)/y^2\right\|_{L^2}^2&\le\int_{-\infty}^t\frac{t^4}{y^4}\d
y=|t|\int_{-\infty}^{-1}x^{-4}\d x\lesssim |t|.
\end{align*}
and consequently
$\|\tilde g_s^{(2)} -\tilde
g_t^{(2)}\|_{2,\PP_\Delta}\lesssim\eta^{1/2}$ for $-\eta\le s,t\le0$.

Having treated these three cases we can show $N(\eps,
Q_n,L^2(\PP_\Delta))\lesssim \eps^{-7}$. For an integer $J>0$ we
consider the grid of points $t_j=-j J^{-6}$ with $j=J^4, J^4+1, J^4+2,\dots,
J^{7}$. We take $\eta=J^{-2}$. By Case~3 we see that $\|\tilde
g_s^{(2)}-\tilde
g_t^{(2)}\|_{2,\PP_\Delta}\lesssim J^{-1}$ for all $s,t\in[-J^{-2},0]$. By
Case~2 we
have
$\|\tilde g_s^{(2)}-\tilde
g_t^{(2)}\|_{2,\PP_\Delta}\lesssim|s-t|^{1/2}/\eta\le J^{-1}$ for
$s,t\in[-(j+1)J^{-6},-jJ^{-6}]$. And by Case~1 $\|\tilde g_s^{(2)}-\tilde
g_t^{(2)}\|_{2,\PP_\Delta}\lesssim J^{-1}$ for $s,t\le-J$.

We have polynomial covering numbers and it suffices
for condition~\eqref{entropy_bound} that
\begin{align*}
 \frac{nr_n^2}{M_n^2\log\left({r_n}^{-1}\right)}\to\infty.
\end{align*}
In \eqref{envelopes} we have seen that $M_n\lesssim
\Delta^{-1/2}\log(1/\Delta)$.
For the choice $r_n=\log(1/\Delta)^{-\alpha}$ we obtain
\begin{align*}
 \frac{n
r_n^2}{M_n^2\log(1/r_n)}\gtrsim\frac{n\Delta}{
\log(1/\Delta)^{2+2\alpha}\log(\log(1/\Delta)) },
\end{align*}
which tends to infinity by $\log^4(1/\Delta)=o(n\Delta)$.

\eqref{uniform_entropy_integrals} In \eqref{entropy_bound} we have seen that
the covering numbers $N(\eps,Q_n,L^2(\PP_\Delta))$ are uniformly in
$n$ polynomial in $\eps^{-1}$ so that the condition is satisfied.

\smallskip

\textbf{Acknowledgement.} The authors acknowledge insightful remarks from the
Associate Editor and two anonymous referees that helped to improve the
presentation of the paper.

Financial Support by the Deutsche
Forschungsgemeinschaft via FOR 1735 Structural Inference in Statistics is
gratefully acknowledged.

\bibliography{bib} 

\end{document}